\pdfoutput=1
\documentclass[11pt,a4paper]{amsart}
\usepackage[utf8]{inputenc}
\usepackage[T1]{fontenc}
\usepackage[british]{babel}
\usepackage[babel=true,tracking=true]{microtype}
\usepackage[margin=3cm]{geometry}
\usepackage{amsmath,amsthm,amssymb,bookmark,multicol,pdflscape}
\usepackage[backend=biber,style=alphabetic,giveninits]{biblatex}
\DeclareNameAlias{default}{family-given}
\AtBeginBibliography{\small}

\theoremstyle{plain}
\newtheorem{theorem}{Theorem}[section]
\newtheorem{lemma}[theorem]{Lemma}
\newtheorem{corollary}[theorem]{Corollary}
\newtheorem{proposition}[theorem]{Proposition}

\theoremstyle{definition}
\newtheorem{definition}[theorem]{Definition}
\newtheorem{conjecture}[theorem]{Conjecture}
\newtheorem{example}[theorem]{Example}
\newtheorem{fact}[theorem]{Fact}

\theoremstyle{remark}
\newtheorem{remark}[theorem]{Remark}
\newtheorem{notation}[theorem]{Notation}
\newtheorem*{acknowledgements}{Acknowledgements}
\newtheorem*{structure}{Structure of the paper}

\title[On Arithmetic Mirror Symmetry for smooth Fano fourfolds]{On Arithmetic Mirror Symmetry \\ for smooth Fano fourfolds}
\author{Mikhail Ovcharenko}
\address
{
  \textnormal{Steklov Mathematical Institute of RAS, 8 Gubkina street, Moscow 119991, Russia.}
  \newline
  \textnormal{HSE University, Laboratory of Mirror Symmetry, 6 Usacheva str., Moscow 119048, Russia.}
}
\email{michael.a.ovcharenko@gmail.com, ovcharenko@mi-ras.ru}

\DeclareMathOperator{\CH}{CH}
\DeclareMathOperator{\codim}{codim}
\DeclareMathOperator{\coker}{coker}
\DeclareMathOperator{\dist}{dist}
\DeclareMathOperator{\even}{even}
\DeclareMathOperator{\FT}{FT}
\DeclareMathOperator{\Gr}{Gr}
\DeclareMathOperator{\Hg}{Hg}
\DeclareMathOperator{\Hom}{Hom}
\DeclareMathOperator{\IH}{IH}
\DeclareMathOperator{\im}{im}
\DeclareMathOperator{\inv}{inv}

\DeclareMathOperator{\OG}{OG}
\DeclareMathOperator{\lcm}{lcm}
\DeclareMathOperator{\Pic}{Pic}
\DeclareMathOperator{\prim}{prim}
\DeclareMathOperator{\rk}{rk}
\DeclareMathOperator{\rf}{rf}
\DeclareMathOperator{\SL}{SL}
\DeclareMathOperator{\Sing}{Sing}
\DeclareMathOperator{\Spec}{Spec}

\begin{document}

\begin{abstract}
  We introduce an explicit class of tempered Laurent polynomials in the sense of Villegas and Doran--Kerr in \(n \leqslant 4\) variables including all Landau--Ginzburg models for smooth Fano threefolds with very ample anticanonical class. We check that it contains Landau--Ginzburg models for various Fano fourfolds which are complete intersections in smooth toric varieties and Grassmannians of planes, or are quiver flag zero loci. We discuss implications to \emph{Arithmetic Mirror Symmetry conjecture}, a Hodge-theoretic approach to the study of Ap\'{e}ry constants of Fano varieties proposed by Golyshev--Kerr--Sasaki. Using the partial case of Arithmetic Mirror Symmetry conjecture proved by Kerr, we provide three examples of a Mirror Symmetry correspondence between specific algebraic classes.
\end{abstract}

\maketitle

\section{Introduction}

A Fano variety is a complex projective variety with ample anticanonical class. They form a crucial part of the classification of projective varieties: according to Minimal Model Program, any smooth projective variety of negative Kodaira dimension should be birationally equivalent to a Mori fibration into Fano varieties (possibly with terminal singularities). Moreover, up to deformation equivalence there are only finitely many Fano varieties of given dimension with bounded singularities.

The classification of smooth Fano varieties is known only in dimension up to~3. One of possible approaches to the classification of higher--dimensional Fano varieties is to apply methods of Mirror Symmetry. Deformation families of Fano varieties conjecturally should correspond to their \emph{Landau--Ginzburg models}, certain quasi-projective varieties equipped with a proper regular function whose general fibre is a Calabi--Yau variety. Algebro-geometric properties of a Fano variety should reflect symplectic properties of its properly compactified Landau--Ginzburg model, and vice versa.

Landau--Ginzburg models for smooth Fano threefolds \(X\) of Picard rank one with very ample \(-K_X\) could be identified with certain rational modular curves via Shioda--Inose construction (see~\cite{golyshev/classification,przyjalkowski/review}). Moreover, Landau--Ginzburg models for smooth Fano threefolds of higher Picard rank together with their deformations provide explicit uniruled structures on moduli spaces of lattice-polarised K3 surfaces (see~\cite{doran/modularity}). However, modularity fails in higher dimensions: for Fano threefolds of Picard rank one the differential Galois group of these families is isomorphic to \(\SL_2(\mathbb{C})\), while this usually fails for Fano varieties of higher dimension (see~\cite{bogner/CY-type,golyshev/classification,coates/periods-fourfolds}). Nevertheless, for many of them their Landau--Ginzburg models have arithmetic properties that admit an interpretation in terms of automorphic forms (see~\cite{golyshev/fibered,gegelia/paramodular}).

In attempts to understand known examples of Landau--Ginzburg models for Fano varieties, the \emph{extremality} property was coined (see Subsection~\ref{subsection:MS} for details). Namely, for Landau--Ginzburg models of Fano varieties of odd dimension the local system of middle-dimensional cohomology of fibres in known cases has zero ``ramification defect''. In general the ``ramification defect'' is expected to be bounded by dimension of the space of middle-dimensional primitive Hodge classes on a Fano variety.

While this partially explains the behaviour of associated local systems, this says nothing about the construction of Landau--Ginzburg models. Conjecturally, a smooth Fano variety admits a degeneration to a toric variety in such a way that a Landau--Ginzburg model after the restriction to an open chart is identified with a Laurent polynomial whose Newton polytope is isomorphic to the fan polytope of a toric degeneration. The constant term series of such a Laurent polynomial is identified with the regularised quantum period of a Fano variety, and the differential operator of minimal order and degree annihilating it is identified with the Picard--Fuchs operator. These Laurent polynomials are not unique: one can produce new ones via mutations, special birational transformations of tori preserving the constant term series of a Laurent polynomial.

For smooth Fano threefolds with very ample anticanonical class their Landau--Ginzburg models could be constructed from Minkowski polynomials, special Laurent polynomials with reflexive Newton polytope whose faces admit Minkowski decomposition into segments or triangles with no interior lattice points, and whose edge polynomials are of the form \((1 + t)^n\) (the latter property is called the \emph{binomial principle} in~\cite{przyjalkowski/review}). Any two Minkowski polynomials with the same period series could be connected by a sequence of mutations (see~\cite{akhtar/minkowski}). If the anticanonical class of a Fano variety is not very ample, the Newton polytope of the Laurent polynomial is no longer reflexive, so Minkowski approach is not applicable. Nevertheless, the binomial principle holds in more general situations, see~\cite{przyjalkowski/review}.

We do not know what the analogue of Minkowski polynomials should look like for smooth Fano fourfolds. Unexpectedly, some insights into this come from the study of logarithmic Mahler measures and special values of L-functions (see~\cite{villegas/mahler,doran/k-theory}). The reason for this is that Laurent polynomials with similar properties naturally arise in this setting as well. Namely, they also have reflexive Newton polytope, and the family of fibres should admit a fibre-wise compactification \(Y \rightarrow \mathbb{A}^1\) whose general fibre is a Calabi--Yau variety. Moreover, these families are \emph{tempered} (see Section~\ref{section:proof-amenable} or~\cite{doran/k-theory}), i.e., the toric-coordinate symbol \(\{x_1, \ldots, x_n\} \in \CH^{n}(\mathbb{G}_m^n, n)_\mathbb{Q}\) completes to a class in the higher Chow group \(\CH^{n}(Y, n)_\mathbb{Q}\) of the total space. According to~\cite{doran/k-theory}, tempered families of Calabi--Yau varieties have arithmetic properties similar to rational modular curves, and the toric-coordinate symbol in \(\CH^{n}(Y, n)_\mathbb{Q}\) is an analogue of Beilinson's Eisenstein symbols (see~\cite[Sections~7,10]{doran/k-theory}).

Landau--Ginzburg models for Fano varieties with very ample anticanonical class of small dimension are tempered. Actually, for a Laurent polynomial in two variables with reflexive Newton polygon temperedness holds if and only if all edge polynomials vanish only at roots of unity (see~\cite{villegas/mahler}). In particular, Landau--Ginzburg models for smooth del Pezzo surfaces with very ample anticanonical class are tempered because they satisfy the binomial principle (see~\cite{przyjalkowski/review}). In dimension 3 Minkowski polynomials are also tempered (see~\cite[Proposition~2.4]{daSilva/arithmetic}). Moreover, any tempered Laurent polynomial is defined over \(\overline{\mathbb{Q}}\) up to scale (see~\cite[Proposition~4.16]{doran/k-theory}).

Four-dimensional tempered Laurent polynomials provide special families of Calabi--Yau threefolds. It is known that unlike the three-dimensional case, i.e., the case of families of K3 surfaces, tempered families of Calabi--Yau threefolds are never ``modular'' (see~\cite[\nopp 10.6]{doran/k-theory} for a precise statement). So it is natural to ask whether Landau--Ginzburg models continue to be tempered for Fano varieties of higher dimension. In the paper we provide a partial positive answer to this in dimension~\(4\).

Namely, we introduce an explicit class of Laurent polynomials in \(n \leqslant 4\) variables generalising Minkowski polynomials, and show that they are tempered, which is the main result of the paper (see Section~\ref{section:proof-amenable} for definitions and details).

\begin{theorem}[see Proposition~\ref{proposition:amenable}]\label{theorem:amenable}
  Let \(\mathsf{p} \in \mathbb{C}[x_1^{\pm}, \ldots, x_n^{\pm}]\), where \(n \leqslant 4\), be an \emph{amenable} (respectively, \emph{strictly amenable}) Laurent polynomial. For \(n = 4\) we also assume that \(\mathsf{p}\) is defined over a totally real abelian number field. It admits a weak (respectively, smooth) log Calabi-Yau compactification \(\mathsf{f} \colon \mathcal{W} \rightarrow \mathbb{P}^1\), and this family is \emph{tempered}, i.e., the toric-coordinate symbol \(\{x_1, \ldots, x_n\}\) in \(\CH^{n}(\mathbb{G}_m^n, n) \otimes_{\mathbb{Z}} \mathbb{Q}\) completes to a motivic cohomology class in \(H^n_{\mathcal{M}}(\mathcal{W} \setminus \mathcal{W}_0, \mathbb{Q}(n))\), where \(\mathcal{W}_0\) is the fibre of \(\mathsf{f}\) over infinity.
\end{theorem}

\begin{proposition}[see~\cite{ovcharenko/supplement}]\label{proposition:examples}
  \begin{enumerate}
  \item Landau--Ginzburg models for Fano fourfold quiver flag zero loci from the list in~\cite{kalashnikov/quiver-laurent} are amenable in 35 of 99 cases, and for 62 of 64 bad cases we construct a mutation-equivalent amenable replacement.
  \item For Fano fourfold complete intersections in smooth toric varieties from the list in~\cite{coates/toric} Landau--Ginzburg models are amenable in 439 of 738 cases, and for 160 of 299 bad cases we construct a mutation-equivalent amenable replacement.
  \item In all cases two-dimensional faces of Newton polytopes admit a lattice Minkowski decomposition into segments and triangles of lattice width 1.
  \end{enumerate}
\end{proposition}

\begin{remark}
  This is a purely computational statement. Here amenable (not strictly amenable) means that we have not checked the existence of a certain nice triangulation of the dual polytope to the Newton polytope (see Subsection~\ref{subsection:toric}). Note that all Laurent polynomials from Proposition~\ref{proposition:examples} are \emph{rigid maximally mutable} (see~\cite{coates/laurent}). According to~\cite{coates/laurent}, this class should describe all Laurent polynomials naturally arising as Landau--Ginzburg models for Fano varieties (see~\cite{kasprzyk/laurent} for further details).

  We expect that all remaining bad cases could also be fixed by appropriate mutations: we have not done this only for computational reasons. Note that amenable Laurent polynomials are not closed under mutations (neither are Minkowski polynomials, see~\cite{coates/laurent}).
\end{remark}

The study of tempered families of Calabi--Yau varieties and their arithmetic properties is largely motivated by the study of \emph{Ap\'{e}ry constants} of a Fano variety
(see~\cite{doran/k-theory,daSilva/arithmetic,kerr/motivic,bloch/gamma,kerr/unipotent,golyshev/apery}). These are asymptotic invariants of its quantum differential equation. In known cases they are period numbers, and could be studied through Mirror Symmetry for Fano varieties. Their study goes all the way back to Ap\'{e}ry's proof of irrationality of \(\zeta(3)\): the famous Ap\'{e}ry recurrence is equivalent to the Picard--Fuchs equation of a Landau--Ginzburg model for a Fano threefold \(V_{12}\) (see Section~\ref{section:apery}). There exists a special subclass of Ap\'{e}ry constants that we call \emph{principal}: they arise from homogeneous solutions of the quantum differential equation.

On the symplectic side the properties of principal Ap\'{e}ry constants are naturally explained by Gamma conjectures of Galkin--Golyshev--Iritani (see~\cite{galkin/gamma,galkin/gamma-mirror}). On the algebraic side Golyshev--Kerr--Sasaki proposed \emph{Arithmetic Mirror Symmetry conjecture} (see~\cite{golyshev/apery}), a general Hodge-theoretic approach to Ap\'{e}ry constants connecting them with limiting behaviour of families of higher Chow cycles on a Landau--Ginzburg model, and applied it to smooth Fano threefolds that are complete intersections in generalised Grassmannians. This approach is heavily based upon the temperedness property. Theorem~\ref{theorem:amenable} and Proposition~\ref{proposition:examples} imply that for a large number of smooth Fano fourfolds their Landau--Ginzburg models are also tempered, so it's natural to expect that the temperedness property holds in general.

Unfortunately, a proper treatment of Ap\'{e}ry constants, Golyshev--Kerr--Sasaki approach, and relevant Hodge-theoretic notions would be too long to be provided here. For this reason we refer the reader to Section~\ref{section:apery} for a short review.

We do not know how to generalise the approach of~\cite{golyshev/apery} to smooth Fano fourfolds yet we can test its predictions for certain Fano fourfold complete intersections in Grassmannians and their Landau--Ginzburg models of BCFKS type (see~\cite{batyrev/grassmannians-MS,batyrev/grassmannians-MS-2}). Following~\cite{przyjalkowski/grassmannian-planes,przyjalkowski/grassmannians}, we treat them as Laurent polynomials. We check that they are strictly amenable, hence tempered, and describe Ap\'{e}ry constants using the proof of \emph{Gamma conjecture~I'} for Grassmannians (see~\cite{galkin/gamma-mirror}).

\begin{notation}
  Let \(t\) be the coordinate on \(\mathbb{P}^1\), where \(t = 0\) corresponds to the fibre of the potential over \(\infty\). We denote by \(\mathcal{H}^3 = \mathcal{H}_f^3 \oplus \mathcal{H}_v^3\) the local system of middle-dimensional cohomology of fibres, where \(\mathcal{H}_f^3\) is the fixed part, and \(\mathcal{H}_v^3\) has no global sections. For any \(p \in \mathbb{P}^1\) we denote by \(T_p\) the monodromy operator around \(p\).

  If \(L = \sum_{i,j} \beta_{ij} t^i D^j\), \(D = t \frac{\partial}{\partial t}\), is a Picard--Fuchs operator, we refer to maximal index \(j\) or \(i\) such that \(\beta_{ij} \neq 0\) for some other \(i\) or \(j\) as its \emph{order} and \emph{degree}, respectively.
\end{notation}

\begin{theorem}\label{theorem:apery}
  The following assertions hold.
  \begin{enumerate}
  \item Let \(V\) be a smooth complete intersection of multidegree \(\eta\) in \(\Gr(2, N)\), where
    \begin{multicols}{3}
      \begin{enumerate}
      \item[(A.1)] \(\eta = (1, 3)\), \(N = 5\);
      \item[(A.2)] \(\eta = (1, 1)\), \(N = 5\); \\
      \item[(B.1)] \(\eta = (2, 2)\), \(N = 5\);
      \item[(B.2)] \(\eta = (1, 2)\), \(N = 5\); \\
      \item[(C.1)] \(\eta = (1^{(3)}, 2)\), \(N = 6\);
      \item[(C.2)] \(\eta = (1^{(4)})\), \(N = 6\);
      \item[(D)] \(\eta = (1^{(6)})\), \(N = 7\).
      \end{enumerate}
    \end{multicols}
    A Landau--Ginzburg model for \(V\) of BCFKS type from~\cite{przyjalkowski/grassmannian-planes,przyjalkowski/grassmannians} is amenable up to some mutation. Moreover, it is strictly amenable for \((A.1)\)--\((C.1)\) and \((D)\).
  \item For \((A)\)--\((C)\), the degree of the Picard--Fuchs operator in \(t\) equals \(2 i_V\), where \(i_V\) is the Fano index. There exists a unique principal Ap\'{e}ry constant \(\zeta(2)\), and there are no non-principal Ap\'{e}ry constants. We have \(\IH^1(\mathbb{P}^1 \setminus \{\infty\}, \mathcal{H}_v^3) \simeq \mathbb{Q}(-2)\) if \(i_V = 1\).
  \item For \((D)\) there exists a principal Ap\'{e}ry constant \(\zeta(2)\) and a non-principal Ap\'{e}ry constant \(\zeta(4)\). The monodromy at \(t = 0,\infty\) is maximal unipotent. The degree of Picard--Fuchs operator is \(5\), and \(\rk(\IH^1(\mathbb{P}^1 \setminus \{\infty\}, \mathcal{H}_v^3)) = 2\). Apart from \(0\) and \(\infty\), there are \(3\) singular points \(\sigma\): for them we have \(\rk(T_{\sigma} - I) = 1\).
  \end{enumerate}
\end{theorem}

\begin{remark}
  We consider Ap\'{e}ry constants only up to rational multiples.
\end{remark}

Theorem~\ref{theorem:apery} states that Landau--Ginzburg models for these Fano fourfolds admit a smooth log-Calabi--Yau compactification and are tempered, which agrees with Golyshev--Kerr--Sasaki approach. Let us discuss its predictions on Ap\'{e}ry constants.

Recall that local systems of middle-dimensional cohomology of fibres on Landau--Ginzburg models of Fano threefolds are extremal. In even dimensions this may fail, and our examples are of ``ramification defect one'': we have \(\IH^1(\mathbb{P}^1, \mathcal{H}_v^3) \simeq \mathbb{Q}(-2)\). The approach of~\cite{golyshev/apery} predicts the existence of a class in \(\CH^2(\mathcal{W} \setminus \mathcal{W}_{\infty}, 0)_{\mathbb{Q}}\) homologous to \(0\) on fibres whose normal function produces the principal Ap\'{e}ry constant \(\zeta(2)\). Actually, at least for the cases \((A.1)\)--\((C.1)\) this could be derived from Kerr's proof of a partial case of Arithmetic Mirror Symmetry conjecture (see~\cite[Theorem~10.10]{kerr/unipotent}).

\begin{corollary}\label{corollary:apery}
  Let \(V\) be a smooth Fano fourfold from Theorem~\ref{theorem:apery} for \((A.1)\)--\((C.1)\), and \(\mathsf{f} \colon \mathcal{W} \rightarrow \mathbb{P}^1\) be its log-Calabi--Yau compactified Landau--Ginzburg model. Then there exists an algebraic cycle in \(\CH^2(\mathcal{W})_{\mathbb{Q}}\) homologous to \(0\) on fibres producing a truncated normal function which is a holomorphic function at \(t = 0\) of the form
\[
  f = \sum_{n = 0}^{\infty} (- B_n + f(0) A_n) t^n; \quad
  f(0) = \lim_{n \to \infty} \frac{B_n}{A_n} = \zeta(2); \quad
  A_0 = 1, \; A_n \in \mathbb{Z}_{> 0}, \; B_0 = 0,
\]
where \((A_n)\) and \((B_n)\) are rational solutions of the regularised quantum recurrence.
\end{corollary}

Note that Corollary~\ref{corollary:apery} could be considered as a Mirror Symmetry correspondence \emph{at the level of specific algebraic cycles}. Recall that for \((A.1)\)--\((C.1)\) there exists an Ap\'{e}ry constant \(\zeta(2)\) which is \emph{principal}, i.e., the restriction of the unique primitive cohomology class of codimension 2 on \(\Gr(2, N)\) to \(V\) produces a holomorphic function
\[
  \widehat{f} = \sum_{n = 0}^{\infty} (- B_n + \widehat{f}(0) A_n) \frac{t^n}{n!}; \quad
  \widehat{f}(0) = \lim_{n \to \infty} \frac{B_n}{A_n} = \zeta(2); \quad
  A_0 = 1, \; A_n \in \mathbb{Z}_{> 0}, \; B_0 = 0,
\]
where \((A_n), (B_n)\) are defined by Givental's J-function (see~\cite{galkin/gamma,galkin/gamma-mirror} for details). Note that \(\widehat{f}\) is a homogeneous solution of the quantum differential equation, and \(f\) is an inhomogeneous solution of the regularised quantum differential equation.

It is interesting to compare this with the case of Fano threefolds, where Mirror Symmetry correspondence ``on the level of algebraic cycles'' also makes sense. Recall that for a smooth Fano threefold \(X\) its general anticanonical section is a smooth K3 surface endowed with a natural lattice polarisation by restriction of the Picard lattice. Meanwhile, a general fibre of a Landau--Ginzburg model for \(X\) is a smooth K3 surface endowed with polarisation by the lattice \(L\) of monodromy invariants. Note that by Deligne's theorem on invariant cycles these are precisely the image of restriction of \(H^2(\mathcal{W}, \mathbb{Z})\) to a general fibre. According to~\cite{doran/modularity}, these lattice polarisations are \emph{Dolgachev--Nikulin dual}, i.e., there exists a lattice isometry \(\Pic(X) \oplus H \simeq L^{\perp}\) inside the second integral cohomology lattice of a K3 surface.

In our situation we have a unique primitive Hodge class in middle cohomology of a Fano fourfold, it vanishes on a general anticanonical section. Corollary~\ref{corollary:apery} shows that it is mirror-dual to a class in \(\CH^2(\mathcal{W})_{\mathbb{Q}} \simeq H^4(\mathcal{W}, \mathbb{Q}(2))\) homologous to \(0\) on fibres. In other words, this situation is a complete opposite to the case of Fano threefolds. Note also that known results from Arithmetic Mirror Symmetry primarily deal with the Picard rank one case and do not consider deformations of Landau--Ginzburg models. In contrast, Landau--Ginzburg models for Fano varieties of higher Picard rank may have exotic properties in comparison with their general deformations (see~\cite{doran/modularity,lee/rationality}).

As we have \(\rk(\IH^1(\mathbb{P}^1 \setminus \{\infty\}, \mathcal{H}_v^3)) = 2\), the approach of~\cite{golyshev/apery} also predicts that there should exist a morphism of mixed Hodge structures \(\mathbb{Q}(-a) \hookrightarrow \IH^1(\mathbb{P}^1 \setminus \{\infty\}, \mathcal{H}_v^3)\) defining a normal function that produces the non-principal Ap\'{e}ry constant \(\zeta(4)\). Yet the situation is more complicated because this inclusion is potentially non-splittable. Another problem is that for Fano threefolds in~\cite{golyshev/apery} we have \(d = 2  i_V\) singular points (apart from \(t = 0,\infty\)), where \(d\) is degree of the Picard--Fuchs operator in \(t\). In these points we have \(\rk(T_{\sigma} - I) = 1\). Our examples \((A)\)--\((C)\) show this behaviour: there are \(d\) singular points with \(\rk(T_{\sigma} - I) = 1\). For \((D)\) we have \(d = 5\), but there are three singular points, though still with \(\rk(T_{\sigma} - I) = 1\). In other words, this violates the assumption of~\cite[Remark~5.2]{golyshev/apery} ensuring the canonical lifting of a normal function. These two problems are connected: see Remark~\ref{remark:calculations}.

\begin{structure}
  In Section~\ref{section:apery} we review Ap\'{e}ry constants and briefly explain the approach of~\cite{golyshev/apery}. In Section~\ref{section:preliminaries} we provide preliminaries on Mirror Symmetry and higher Chow cycles. In Section~\ref{section:proof-amenable} and~\ref{section:proof-apery} we prove Theorems~\ref{theorem:amenable},~\ref{theorem:apery} and Corollary~\ref{corollary:apery}. In Appendices~\ref{appendix:local-systems},~\ref{appendix:laurent-polynomials} we provide main calculations, the rest could be found at~\cite{ovcharenko/supplement}.
\end{structure}

\begin{acknowledgements}
  The article was prepared within the framework of the project ``International academic cooperation'' HSE University. The author is grateful to C.~Doran, V.~Golyshev, A.~Harder, M.~Kerr, and V.~Przyjalkowski for valuable discussions, careful reading of the paper and many important corrections.
\end{acknowledgements}

\section{Ap\'{e}ry constants of Fano varieties}\label{section:apery}

In 1979 Roger Ap\'{e}ry proved his famous result on irrationality of the Ap\'{e}ry constant~\(\zeta(3)\) by presenting it as a limit of rational numbers of the form
\[
  \lim_{n \to \infty} B_n \cdot A_n^{-1} = \zeta(3),
\]
where \((A_n)\) and \((B_n)\) are solutions to the Ap\'{e}ry recurrence
\begin{equation}
  n^3 u_n - (34 n^3 - 51 n^2 + 27 n - 5) u_{n - 1} + (n - 1)^3 u_{n - 2} = 0
\end{equation}
with initial conditions \(A_0 = 1, A_1 = 5\) and \(B_0 = 0, B_1 = 6\), correspondingly. Irrationality is implied by the following remarkable properties (see~\cite[\nopp 1]{zagier/arithmetic}):
\begin{itemize}
\item[(a)] \(A_n \in \mathbb{Z}\) for all \(n \in \mathbb{Z}_{\geqslant 0}\);
\item[(b)] \(d_n^3 B_n \in \mathbb{Z}\) for all \(n \in \mathbb{Z}_{\geqslant 0}\), where \(d_n = \lcm(1, 2, \ldots, n)\).
\end{itemize}
In a similar way, the small Ap\'{e}ry constant \(\zeta(2)\) can be presented as the limit
\[
  \lim_{n \to \infty} \widetilde{B}_n \cdot \widetilde{A}_n^{-1} = \zeta(2),
\]
where \((\widetilde{A}_n)\) and \((\widetilde{B}_n)\) are solutions to the small Ap\'{e}ry recurrence
\begin{equation}
  n^2 u_n - (11 n^2 - 11 n + 3) u_{n - 1} - (n - 1)^2 u_{n - 2} = 0
\end{equation}
with the initial conditions \(\widetilde{A}_0 = 1, \widetilde{A}_1 = 3\) and \(\widetilde{B}_0 = 0, \widetilde{B}_1 = 5\), correspondingly.

The solutions \((A_n)\) and \((\widetilde{A}_n)\) could be explicitly presented in the following form:
\[
  A_n = \sum_{k = 0}^n \binom{n}{k}^2 \binom{n + k}{k}^2, \quad
  \widetilde{A}_n = \sum_{k = 0}^n \binom{n}{k}^2 \binom{n + k}{k}.
\]
Note that the existence of ``Ap\'{e}ry-like'' rational linear recurrences admitting integral solutions is extremely rare (see~\cite{zagier/arithmetic}). So this should be considered as a manifestation of some deep geometric or arithmetic phenomenon. More precisely,
\begin{enumerate}
\item Beukers has derived the properties (a), (b) as a consequence of his formula
  \[
    B_n - A_n \zeta(3) = \int_0^1 \int_0^1 \int_0^1
    \left ( \frac{x (1 - x) y (1 - y) z (1 - z)}{1 - z + x y z} \right )^n
    \frac{dx \wedge dy \wedge dz}{1 - z + x y z}.
  \]
\item Beukers' formula was later reinterpreted by Beukers and Peters in terms of periods of a certain one-dimensional family of K3 surfaces. Put
  \[
    A(t) = \sum_{n = 0}^{\infty} A_n t^n, \quad
    B(t) = \sum_{n = 0}^{\infty} B_n t^n.
  \]
  The Ap\'{e}ry recurrence is equivalent to \(\mathcal{L}(A(t)) = 0\), where \(\mathcal{L}\) is the differential operator of minimal order and degree annihilating the series \(A(t)\):
  \[
    \mathcal{L} = D^3 - t (34 D^3 + 51 D^2 + 27 D + 5) + t^2 (D + 1)^3, \quad
    D = t \frac{d}{dt}.
  \]
  Moreover, we also have the identity \((D - 1) \mathcal{L} B(t) = 0\). Beukers and Peters showed that the operator \(\mathcal{L}\) can be interpreted as the Picard--Fuchs differential operator associated with the family of surfaces
  \[
    V_t \colon \frac{x (1 - x) y (1 - y) z (1 - z)} {1 - z + x y z} = \frac{1}{t}, \quad t \in \mathbb{C},
  \]
  where \(V_t\) is birationally equivalent to a K3 surface with Picard number 19 for a general choice of \(t\) (compare this with the Beukers' formula).
\item After a birational change of variables \(x = (x + z - 1) / yz\) the family \(V_t\) is given by the equation \(F(x, y, z) = 1/t\), where \(F(x, y, z)\) is the Laurent polynomial
  \[
    F(x, y, z) = (y - 1)(z - 1)(x + z - 1) (y z - x - z + 1) \cdot (x y z)^{-1}.
  \]
  In these terms the series \(A(t)\) coincides with the series \(\sum c_n t^n\), where \(c_n\) is the constant term of the Laurent polynomial \(F(x, y, z)^n\).
\item At last, the series \(A(t)\) can be interpreted as the regularised quantum period of the smooth Fano threefold \(V_{12}\), i.e., of a linear section of \(\OG(5, 10)\) of codimension 7 (see~\cite{golyshev/tate,golyshev/satake}). It is a non-negative integer sequence defined in terms of the small quantum cohomology ring of a Fano variety.
\end{enumerate}

These approaches are interconnected. Namely, the Beukers--Peters pencil of K3 surfaces from~(2) up to a birational transformation arises from fibres of the Laurent polynomial from~(3). Moreover, this Laurent polynomial is a \emph{Landau--Ginzburg model} for the smooth Fano variety \(V_{12}\): for the moment, this means that the power series from~(3) and~(4) coincide. Equivalently, the Picard--Fuchs equation from~(2) coincides with the \emph{regularised quantum differential equation} of the Fano threefold \(V_{12}\) (see~\cite{golyshev/classification,golyshev/DN-equations}).

This is a form of Mirror Symmetry for Fano varieties. Recall that Mirror Symmetry associates to a Fano variety its Landau--Ginzburg model, which is a one-dimensional family of Calabi--Yau varieties such that a general fibre of the family is mirror-dual to a general anticanonical section of a Fano variety (see Subsection~\ref{subsection:MS} for details). There are many different but interrelated forms of Mirror Symmetry for Fano varieties. For the moment, its most important aspect is that the Picard--Fuchs equation for a Landau--Ginzburg model should coincide with the regularised quantum differential equation of a Fano variety (see Section~\ref{section:preliminaries} or~\cite{przyjalkowski/review,galkin/gamma-mirror,kasprzyk/laurent} for more details).

The Picard--Fuchs equation \(\mathcal{L}(\sum c_i s^i) = 0\), where \(s = t^{i_X}\), and \(i_X\) is the Fano index, yields the \emph{regularised quantum recurrence} on the coefficients \(c_i\) (see Subsection~\ref{subsection:MS} for details). Let us consider the \(\mathbb{Q}\)-linear space of its solutions. Note that not every solution of the recurrence defines a solution of the differential equation: there could exist ``almost solutions'' \(\Phi(s)\) such that \(\mathcal{L}(\Phi(s))\) is not zero, but a polynomial \(\Psi(s)\) of a small degree. Consequently, they are either solutions of the usual Picard--Fuchs equation, or they solve inhomogeneous Picard--Fuchs equations \(\mathcal{L}(\cdot) = \Psi(s)\).

It is well-known that the homogeneous Picard--Fuchs equation governs periods of families of algebraic varieties. In turn, inhomogeneous Picard--Fuchs equations govern the behaviour of families of classical or higher algebraic cycles (see~\cite{delAngel/equations}). Note that (higher) algebraic cycles always yield only \emph{inhomogeneous} non-zero solutions of the Picard--Fuchs equation (see~\cite{delAngel/equations} or~\cite[\nopp 1.7]{kerr/exponential}). Meanwhile, the space of solutions to the non-regularised quantum differential equation could be identified with the space of ``coprimitive classes'' on a Fano variety, i.e., with \(\ker(\cap c_1(X))\) inside \(H_{\bullet}(X) := H_{\even}(X, \mathbb{C})\) (see Subsection~\ref{subsection:MS} for details). Their formal inverse Laplace transforms form a subspace of solutions to the regularised quantum recurrence.

Let us choose a basis in the \(\mathbb{Q}\)-solution space of the regularised quantum recurrence. There exists a canonical integer solution \((A_n)\) with \(A_0 = 1\), the regularised quantum period (see Subsection~\ref{subsection:MS}). Moreover, Mirror Symmetry predicts that the monodromy around \(t = 0\) is maximally unipotent (see~\cite{cheltsov/review}). If a smooth Fano variety \(V\) is \emph{quantum minimal} (see~\cite{galkin/minimal,przyjalkowski/LG-models}), i.e., the order of the Picard--Fuchs operator equals \(r = \dim(V)\), then its indicial equation is of the form \(T^r = 0\), and by Frobenius method there are no other power series solutions apart from \((A_n)\) (see~\cite{vanStraten/CY-operators}). Let \((B_n^{(i)})\) be other solutions of the recurrence completing \((A_n)\) to a basis.

\begin{definition}
  The \emph{Ap\'{e}ry constants} of a Fano variety are real numbers of the form
  \[
    \alpha_i = \lim_{n \to \infty} B_n^{(i)} \cdot A_n^{-1}.
  \]
  They are \emph{principal} if the formal Laplace transform of the solution of the recurrence solves the quantum differential equation, and are \emph{non-principal} otherwise.
\end{definition}

It is not obvious that Ap\'{e}ry constants exist. Symplectic geometry answers this for principal Ap\'{e}ry constants: they exist assuming \emph{Conjecture~\(\mathcal{O}\)} (see~\cite[Theorem~3.7.1]{galkin/gamma}). A remarkable property is their functoriality: Ap\'{e}ry constants of an ambient Fano variety arise as Ap\'{e}ry constants for Fano complete intersections therein. This is a consequence of quantum Lefschetz principle (see~\cite[Section~8]{galkin/gamma-mirror}). Note that principal Ap\'{e}ry constants for one Fano variety become non-principal for another (see~\cite{golyshev/tate,golyshev/apery}) if the corresponding primitive class (under Poincar\'{e} and hard Lefschetz dualities) vanishes under restriction. It is not clear that \(\alpha_i\) are period numbers: principal Ap\'{e}ry constants are \(\mathbb{Q}\)-linear polynomials in zeta values assuming \emph{Gamma conjecture~I'} (see~\cite{galkin/gamma}).
This is known to be true for Grassmannians (see~\cite[Theorem~9.20]{galkin/gamma-mirror}). Moreover, Gamma conjecture I holds for del Pezzo surfaces (see~\cite{hu/gamma}) and Fano threefolds of Picard rank one (see~\cite{golyshev/prime}). Not all Ap\'{e}ry constants are of this form: a smooth Fano threefold~\(V_{18}\) has an Ap\'{e}ry constant \(L(\chi_3, 3)\) (see~\cite{golyshev/tate}). Note that in general both Conjecture~\(\mathcal{O}\) and Gamma conjecture~I may fail as stated and should be properly modified (see~\cite{galkin/gamma-modification}).

Golyshev--Kerr--Sasaki proposed in~\cite{golyshev/apery} (see also~\cite{kerr/motivic,kerr/unipotent,bloch/gamma}) \emph{Arithmetic Mirror Symmetry conjecture}, a general Hodge-theoretic approach to Ap\'{e}ry constants. It aims to present them as special values of \emph{truncated (higher) normal functions} (see~\cite{kerr/exponential,delAngel/equations}) of families of (higher) cycles on Landau--Ginzburg models; in particular, as period numbers. These functions are \(\mathbb{C}\)-linear solutions of the regularised quantum recurrence, hence they are \(\mathbb{C}\)-linear combinations of the solution basis. Very informally, the goal is to show that the series like
\[
  F(s) = A(s) \zeta(3) - B(s) = \sum_{i = 0}^{\infty} c_i s^i, \quad
  F(0) = A(0) \zeta(3) - B(0) = \zeta(3),
\]
from the section opening ``are of geometric origin'': they arise as Taylor series for certain holomorphic functions such that \(f(0) = \lim_{n \rightarrow \infty} (B_n / A_n)\), and they are associated with families of classical or higher algebraic cycles on Landau--Ginzburg models.

\begin{conjecture}[\cite{golyshev/apery}]\label{conjecture:arithmetic-MS}
  For each smooth Fano \(n\)-fold \(X\) admitting a toric degeneration its Ap\'{e}ry numbers arise as limits at \(t = \infty\) of classical or higher truncated normal functions produced by cycles on a 1-parameter family of Calabi--Yau \((n-1)\)-folds defined over \(\overline{\mathbb{Q}}\), together with extension classes in the monodromy-invariant part of a limiting mixed Hodge structure of the family.
\end{conjecture}

Despite the terminology a \emph{normal function} is not a function but a section of a fibration into (higher) intermediate Jacobians (see~\cite{kerr/exponential} for a review of (higher) normal functions). A normal function could be considered as a highly multi-valued holomorphic section of a vector bundle associated with the corresponding local system. Choosing a branch and pairing it with a family of differential forms, one obtains a \emph{truncated normal function}, which solves an inhomogeneous Picard--Fuchs equation (see~\cite{delAngel/equations}). Under certain assumptions a truncated normal function may be chosen canonically (see~\cite[Remark~5.2]{golyshev/apery}, and~\cite{delAngel/elevator}).

Conjecture~\ref{conjecture:arithmetic-MS} as stated is rather vague: it is more a guiding principle than a strict statement. Under restrictive assumptions it could be made precise (see~\cite[Sections~5--6]{golyshev/apery}). This version was proved in~\cite{golyshev/apery} for smooth Fano threefolds \(V_{2N}\) for \(N = 5, \ldots, 9\). Moreover, the techniques developed in~\cite{doran/k-theory,kerr/motivic,kerr/unipotent,golyshev/apery} are partially available for Landau--Ginzburg models for higher-dimensional Fano varieties provided that they have a smooth log Calabi--Yau compactification \(\mathcal{W} \rightarrow \mathbb{P}^1\) (see Subsection~\ref{subsection:MS}) and are tempered (see Section~\ref{section:proof-amenable}), which is a kind of K-theoretic unobstructedness. For such families the main result of~\cite{golyshev/apery} states that morphisms of mixed Hodge structures of the form
\(
  \mathbb{Q}(-a) \hookrightarrow \IH^1(\mathbb{P}^1 \setminus \{\infty\}, \mathcal{H}_v^{n - 1})
\)
induce \emph{admissible} higher normal functions which are \emph{non-singular} away from \(t = \infty\) (see~\cite{kerr/exponential} for relevant terminology), where \(\mathcal{H}_v^{n - 1}\) is the ``variable'' part of the local system. The expectation of~\cite{golyshev/apery} is that all Ap\'{e}ry constants should arise as limits at \(t = \infty\) of their associated truncated (higher) normal functions.
Under stated assumptions Conjecture~\ref{conjecture:arithmetic-MS} is proved (minus the construction of a cycle) in the case when the Picard--Fuchs operator has degree \(2\) in \(t\) (see~\cite[Theorem~10.10]{kerr/unipotent}).

Moreover, according to Beilinson--Hodge conjecture (see~\cite[Conjecture~4.8]{golyshev/apery}), these admissible higher normal functions are induced by families of higher algebraic cycles on the open subset \(\mathcal{W} \setminus \mathcal{W}_{\infty}\) produced by Chow classes in \(\CH^a(\mathcal{W} \setminus \mathcal{W}_{\infty}, 2 a - n)_{\mathbb{Q}}\). Yet it is not clear whether Golyshev--Kerr--Sasaki approach produces non-trivial (higher) normal functions. Some methods for checking this are provided in~\cite[Section~5]{golyshev/apery}, though they are based on constructing explicit solutions to Beilinson--Hodge conjecture. Another problem is existence of a canonical lifting of a normal function to a holomorphic function (see~\cite[Remark~5.2]{golyshev/apery} for a partial solution).

\section{Preliminaries}\label{section:preliminaries}

\subsection{Mirror Symmetry for Fano varieties}\label{subsection:MS}

In this subsection we very briefly review Mirror Symmetry for Fano varieties and properties of Landau--Ginzburg models. We refer the reader to cited literature for a more comprehensive review.

\begin{definition}
  Let \(X\) be a smooth Fano variety, and \(H^{\bullet}(X) = \oplus_{i = 0}^{\dim(X)} H^{2i}(X, \mathbb{C})\) be its cohomology ring (we consider only even degrees). We denote by \(\Lambda = \mathbb{Z}[H_2(X, \mathbb{Z})]\) the \emph{Novikov ring} of \(X\), i.e., the group algebra of the abelian group \(H_2(X, \mathbb{Z})\). The \emph{small quantum cohomology ring} of the Fano variety \(X\) is defined as
  \[
    QH(X) = H^{\bullet}(X) \otimes_{\mathbb{Z}} \Lambda, \quad a \star b = \sum_{\beta \in H_2(X, \mathbb{Z})} (a \star b)_{\beta} q^{\beta},
  \]
  where for all \(C \in H_{\bullet}(X)\) we have \(((a \star b)_{\beta}, C) = \langle A, B, C \rangle_{\beta}\) for Poincar\'{e} duals \(A, B  \in H_{\bullet}(X)\), and \(\langle A, B, C \rangle_{\beta}\) is a three-point genus zero Gromov-Witten invariant of \(X\) (see~\cite{guest/quantum}).
\end{definition}

It is well-known that \(QH(X)\) is an associative algebra. Moreover, by the \emph{dimension axiom} we have \(2 c_1(X) \cdot \beta + 2 \dim(X) = \deg(a) + \deg(b) + \deg(c)\), hence the Fano assumption implies that the sum above is finite. Choose a basis \(\beta_1, \ldots, \beta_r \in H_2(X, \mathbb{Z})\). Then for any tuple \(u = (u_1, \ldots, u_r) \in \mathbb{C}^r\) the substitution \((q^{\beta_1}, \ldots q^{\beta_r}) = u\) gives a deformation \(\star_{u}\) of the multiplication in the cohomology ring \(H^{\bullet}(X)\).

Let us consider the trivial vector bundle \(H^{\bullet}(X) \times H^2(X, \mathbb{C}) \rightarrow H^2(X, \mathbb{C})\) over \(H^2(X, \mathbb{C})\). It is naturally endowed with \emph{Dubrovin connection} \(\nabla\) as follows:
\[
  \nabla_{D_i}(\psi) = D_i(\psi) - \beta^i \star \psi, \quad i = 1, \ldots, r, \quad
  D_i = q_i \frac{\partial}{\partial q_i}, \quad \sum q_i \beta^i \in H^2(X, \mathbb{C}),
\]
for any map \(\psi \colon H^2(X, \mathbb{C}) \rightarrow H^{\bullet}(X)\). The associativity of the quantum product implies that this connection is flat, hence \(\{\nabla_{D_i}\}\) act on the space \(M = \{\psi \colon H^2(X, \mathbb{C}) \rightarrow H^{\bullet}(X)\}\) as differential operators on \(H^2(X, \mathbb{C})\). So the vector space \(M\) can be naturally considered as an algebraic \(\mathcal{D}\)-module \(\mathcal{M}_X\) over the base \(H^2(X, \mathbb{C})\).

\begin{definition}
  \begin{enumerate}
  \item The \(\mathcal{D}\)-module \(\mathcal{M}_X\) is the \emph{(small) quantum \(\mathcal{D}\)-module} of \(X\).
  \item The \emph{anticanonical quantum \(\mathcal{D}\)-module} is its restriction to \(\mathbb{C} \langle -K_X \rangle \hookrightarrow H^2(X, \mathbb{C})\) with respect to the parameter \(s = t^{\langle \beta, -K_X\rangle}\), whose solutions have the form
    \[
      \psi \colon \mathbb{C} \langle -K_X \rangle \rightarrow H^{\bullet}(X), \quad
      D(\psi) = -K_X \star_t \psi, \quad
      D = t \frac{\partial}{\partial t},
    \]
    where \(\star_t\) is obtained from \(\star\) by the substitution \(q^{\beta} = t^{\langle\beta, -K_X\rangle}\). Note that its solutions could be identified with the space of initial conditions \(\psi(0)\) for this matrix linear differential equation, i.e., with \(\{\alpha \in H^{\bullet}(X) : K_X \cup \alpha = 0\}\). Moreover, the canonical holomorphic solution \(J(t)\) with \(J(0) = [pt]\) could be explicitly defined via Givental's J-function (see~\cite[Subsection~5.4]{guest/quantum}).
  \item The \emph{quantum period} of \(X\) is its Poincar\'{e} pairing \(\widehat{A}_X = \langle [pt], J(t) \rangle \) with the class \([pt] \in H_0(X)\), and the \emph{anticanonical quantum differential operator} \(\widehat{L}_X \in \mathcal{D}_{\mathbb{A}^1}\) is the operator of minimal order and degree annihilating it. The anticanonical quantum \(\mathcal{D}\)-module is isomorphic to \(\mathcal{D}_{\mathbb{A}^1} / \widehat{L}_X \mathcal{D}_{\mathbb{A}^1}\) because the operator \(\widehat{L}_X\) could be identified with minimal polynomial of the quantum multiplication operator \(-K_X \star_t \bullet\) over \(\mathbb{Q}[t]\) (see~\cite[Chapter~4]{guest/quantum}).

    The order of \(\widehat{L}_X\) equals the rank over \(\mathbb{Q}[t]\) of the subring \(\langle t, K_X \rangle \subset QH(X)\), and solutions of \(\widehat{L}_X\) are of the form \(\langle [pt], \psi \rangle\), where \(\psi\) is a solution of the anticanonical quantum \(\mathcal{D}\)-module. Equivalently, they are of the form \(\langle \beta, J(t) \rangle\), where \(\beta \in H_{\bullet}(X)\) is such that \(K_X \cap \beta = 0\) (see~\cite[Subsection~5.3]{guest/quantum}).
  \end{enumerate}
\end{definition}

By an abuse of notation, we do not distinguish between the anticanonical quantum \(\mathcal{D}\)-module and its cyclic part. It is known that the anticanonical quantum \(\mathcal{D}\)-module is smooth away from the points \(t = 0\) and \(t = \infty\). It has a regular singularity at \(t = 0\), and an irregular singularity at \(t = \infty\) (see~\cite[Chapter~V]{sabbah/isomonodromic} or~\cite{golyshev/classification}).
There is an operation of \emph{regularisation} producing a \(\mathcal{D}\)-module over \(\mathbb{P}^1\) with regular singularities.

\begin{fact}
  \begin{enumerate}
  \item Let \(L = \sum p_i(t) \partial_t^i\) be a differential operator over \(\mathbb{A}^1\). Its \emph{Fourier transform} is the differential operator \(\FT(L) = \sum p_i(\partial_t) (-t)^i\). We have \(\FT^2(L) = [-1]^*(L)\), where \([-1] \colon \mathbb{A}^1 \rightarrow \mathbb{A}^1\) maps \(t \mapsto -t\).
  \item Put \(\inv \colon \mathbb{G}_m \rightarrow \mathbb{G}_m\), \(t \mapsto t^{-1}\), and note that \(D\) is an operator on \(\mathbb{G}_m\). Then the operator \(L_X = \inv^*(\FT(t^{-1} \widehat{L}_X)) t\) is regular over \(\mathbb{P}^1\) and is called the \emph{regularised quantum differential operator} (see~\cite[Chapter~V]{sabbah/isomonodromic} or~\cite{golyshev/classification}).
  \item The operator \(L_X\) is the operator of minimal order and degree annihilating the \emph{regularised quantum period} \(A_X = \sum i! a_i t^i\), where we put \(\widehat{A}_X = \sum a_i t^i\). Under mild assumptions we have \(A_X \in \mathbb{Z}_{\geqslant 0} [[t]]\) (see~\cite{mandel/periods,johnston/quantum}).
  \end{enumerate}
\end{fact}

\begin{definition}
  Let \(X\) be a Fano variety of Fano index \(i_X\), and let \(L_X = \sum_{i,j} \beta_{ij} t^{i \cdot i_X} D_t^j\) be its regularised quantum differential operator. We can also present it in the form \(L_X = \sum_{i,j} \beta_{ij} i_X^j s^i D_s^j\), where \(s = t^{i_X}\). The \emph{regularised quantum recurrence} is
  \[
    \sum_{i,j} \beta_{ij} (i_X \cdot (n - i))^j u_{n - i} = 0 \quad \forall n \geqslant \deg_s(L_X) =
    \max \left ( \{ i \; \vert \; \beta_{ij} \neq 0 \text{ for some } j \} \right ) .
  \]
\end{definition}

\begin{remark}\label{remark:recurrence}
  Dimension of the space of solutions to the recurrence is \(\deg_s(L_X)\). These solutions define only inhomogeneous solutions \(G(t)\) of the differential equation in the sense that \(L_X(G)\) is a polynomial in \(s = t^{i_X}\) of degree less than \(\deg_s(L_X)\).
\end{remark}

So a Fano variety \(X\) naturally defines a \(\mathcal{D}\)-module over \(\mathbb{A}^1\) with regular singularities. A general expectation of Mirror Symmetry is that this \(\mathcal{D}\)-module is geometric: there should exist a family of Calabi--Yau varieties of dimension \(\dim(X) - 1\) over \(\mathbb{A}^1\) such that its general fibre is mirror-dual to an anticanonical section of a Fano variety, and our \(\mathcal{D}\)-module should underlie the corresponding variation of Hodge structures on middle cohomology.

In practice these families (loosely referred to as \emph{Landau--Ginzburg models}) arise as fibre-wise compactifications of families of fibres of certain Laurent polynomials in \(\dim(X)\) variables (see~\cite{przyjalkowski/review}). The Picard--Fuchs operator could be described as the differential operator of minimal order and degree annihilating the constant coefficient series \(\sum [\mathsf{p}^i]_0 t^i\) of a Laurent polynomial (see~\cite[Theorem~3.4]{przyjalkowski/review}). We refer the reader to~\cite{kasprzyk/laurent} for a review of this approach to Landau--Ginzburg models.

\begin{definition}
  Let \(\mathsf{p} \in \mathbb{C}[x_1^{\pm}, \ldots, x_n^{\pm}]\) be a Laurent polynomial, where \(n = \dim(X)\). We refer to \(\mathsf{p}\) as a \emph{weak Landau--Ginzburg model} if its constant coefficient series \(\sum [\mathsf{p}^i]_0 t^i\) coincides with the regularised quantum period of a Fano variety \(X\).
\end{definition}

Laurent polynomials arising as Landau--Ginzburg models are not unique: one can produce new ones via \emph{mutations}, special birational transformations of algebraic tori preserving the constant coefficient series (see~\cite{akhtar/minkowski} and~\cite[\nopp 3.2]{doran/modularity} for details).

\begin{definition}\label{definition:log-CY}
  Let \(\mathsf{p} \colon \mathbb{G}_m^n \rightarrow \mathbb{C}\) be a Laurent polynomial. Its compactification to \(\mathsf{f} \colon \mathcal{W} \rightarrow \mathbb{P}^1\), where \(\mathcal{W}\) is smooth projective (respectively, \(\mathcal{W}\) is projective with \(\mathbb{Q}\)-factorial terminal toric singularities of codimension \(> 3\), and a general fibre is smooth), and \(-K_{\mathcal{W}} \sim \mathsf{f}^{-1}(\infty)\), is a \emph{(weak) log Calabi--Yau compactification}. A smooth log Calabi--Yau compactification such that \(\mathsf{f}^{-1}(\infty)\) has simple normal crossings is called \emph{tame}.
\end{definition}

\begin{remark}
  For \(n \leqslant 4\) a general fibre is actually smooth by generic smoothness. Moreover, codimension~\(> 3\) assumption is related to the existence of \emph{maximal projective crepant partial desingularisations} of toric varieties (see Section~\ref{section:proof-amenable} or~\cite{batyrev/polyhedra} for details).
\end{remark}

\begin{remark}
  As \(\mathcal{W} \supset \mathbb{G}_m^n\) is smooth or has toric singularities (consequently, in both cases it has rational singularities and any its resolution is a rational variety), we obtain that \(h^q(\mathcal{W}, \mathcal{O}_\mathcal{W}) = 0\) for any \(q \neq 0\). Moreover, Serre duality implies \(h^q(\mathcal{W}, \mathcal{O}_\mathcal{W}(K_\mathcal{W})) = 0\) for \(q \neq n\). The exact sequence of ideal sheaves for (any) fibre \(F \sim -K_\mathcal{W}\) implies \(h^q(F, \mathcal{O}_F) = 0\) for \(0 < q < n - 1\). As we have \(\codim_{\mathcal{W}}(\Sing(\mathcal{W})) > 3\), and a general fibre is smooth, the adjunction formula implies that it is smooth Calabi--Yau.
\end{remark}

\begin{conjecture}\label{conjecture:logCY-compactification}
  A smooth Fano variety with very ample anticanonical class admits a Landau--Ginzburg model with a smooth tame compactification.
\end{conjecture}

\begin{example}[see~\cite{przyjalkowski/review,kasprzyk/laurent}]
  Conjecture~\ref{conjecture:logCY-compactification} holds for smooth del Pezzo surfaces, smooth Fano threefolds, and smooth Fano complete intersections in \(\mathbb{P}^n\).
\end{example}

\begin{definition}[see~{\cite[\nopp 2]{coates/mirror}}]
  Let \(\mathbb{V} = \mathbb{V}_f \oplus \mathbb{V}_v\) be the local system associated with the regularised quantum \(\mathcal{D}\)-module of a smooth Fano variety \(X\), where \(\mathbb{V}_f\) and \(\mathbb{V}_v\) are its fixed and variable parts. The \emph{ramification} \(\rf(\mathbb{V}_v)\) is
  \(
    \rf (\mathbb{V}_v) = \sum_{p \in \mathbb{P}^1} \dim(\mathbb{V}_x / \mathbb{V}_x^{T_p}),
  \)
  where \(T_p\) is the monodromy around \(p \in \mathbb{P}^1\), and \(x \in \mathbb{P}^1\) is a generic point. The \emph{ramification defect} is \(\rf (\mathbb{V}_v) - 2 \rk (\mathbb{V}_v) = \rk(\IH^1(\mathbb{P}^1, \mathbb{V}_v))\). If we have \(2 \rk (\mathbb{V}_v) = \rf (\mathbb{V}_v)\), i.e., \(\IH^1(\mathbb{P}^1, \mathbb{V}_v) = 0\), the local system \(\mathbb{V}_v\) is \emph{extremal}.
\end{definition}

\begin{conjecture}[see~{\cite[Expectation~4.10]{corti/extremal}}]
  Let \(H^{(n/2, n/2)}_{\prim}(X)\) be the space of middle-dimensional primitive Hodge classes on a Fano variety \(X\). The following inequality holds: \(\rk(\IH^1(\mathbb{P}^1, \mathbb{V}_v)) \leqslant \dim(H^{(n/2, n/2)}_{\prim}(X))\). In particular, \(\mathbb{V}_v\) is extremal for odd \(n\).
\end{conjecture}

\subsection{\texorpdfstring{\(\mathsf{M}\)}{M}-polynomials and toric MMP}\label{subsection:toric}

In this subsection we review \emph{Minkowski polynomials} defining Landau--Ginzburg models for smooth Fano threefolds with very ample anticanonical class. Closely following~\cite[\nopp 3.1--3.2]{doran/modularity}, we review a more general formalism of \emph{\(\mathsf{M}\)-polynomials}, which should govern Landau--Ginzburg models for Fano varieties with very ample anticanonical class in higher dimensions.

\begin{definition}
  For a lattice polytope \(\delta\) in a lattice \(M\), a \emph{lattice Minkowski decomposition} \(\mathsf{M}(\delta)\) of \(\delta\) is an unordered list of pairs \(\{(\delta_1,n_1),\ldots, (\delta_k,n_k)\}\), where~\(\delta_i\) is an integral polytope, and \(n_i\) is a positive integer so that
  \[
    \delta = \underbrace{\delta_1 + \cdots + \delta_1}_{n_1} + \cdots + \underbrace{\delta_k + \cdots + \delta_k}_{n_k}.
  \]
  Furthermore, for each \(\delta_i\) let \(\delta_{i,\mathbb{Z}} = \delta_i \cap M\). We also require that
  \[
    \delta \cap M =: \delta_\mathbb{Z} = \underbrace{\delta_{1,\mathbb{Z}} + \cdots + \delta_{1,\mathbb{Z}}}_{n_1} + \cdots + \underbrace{\delta_{k,\mathbb{Z}} + \cdots
      + \delta_{k,\mathbb{Z}}}_{n_k}.
  \]
  Let \(\mathsf{M}_1(\delta) = \{(\sigma_1,m_1),\dots, (\sigma_\ell, m_\ell)\}\) and \(\mathsf{M}_2(\delta) = \{(\delta_1,n_1),\dots, (\delta_k,n_k)\}\) be two lattice Minkowski decompositions of \(\delta\). We say that \(\mathsf{M}_1(\delta)\) \emph{refines} \(\mathsf{M}_2(\delta)\) if there is a partition \(I_1\cup \dots \cup I_k=\{1,\dots, \ell\}\) so that \(m_i \mid n_j\) for all \(i \in I_j\) and if for all \(j\),
  \[
    \sum_{i \in I_j} \left(\dfrac{m_i}{n_j}\right)\sigma_i = \delta_j.
  \]
\end{definition}

For example, any lattice Minkowski decomposition of \(\delta\) refines the trivial decomposition \(\{(\delta,1)\}\), and a lattice Minkowski decomposition \(\{(\delta,1),\dots, (\delta,1)\}\) refines \(\{(\delta,k)\}\). For a polytope \(\delta\), Minkowski decompositions of \(\delta\) are partially ordered by refinement. If \(\mathsf{M}_1(\delta)\) is a refinement of \(\mathsf{M}_2(\delta)\), we say that \(\mathsf{M}_2(\delta) \preceq \mathsf{M}_1(\delta)\). Observe that if \(\delta'\) is a face of \(\delta\), then any lattice Minkowski decomposition of \(\delta\) induces a lattice Minkowski decomposition of \(\delta'\) which we denote \(\mathsf{M}(\delta)|_{\delta'}\).

\begin{definition}
  Given a reflexive polytope \(P\), a \emph{Minkowski datum} for \(P\) is the data of a lattice Minkowski decomposition \(\mathsf{M}(\delta)\) of each face \(\delta\) of \(P\) so that for any pair of face \(\delta' \subseteq \delta\) the induced Minkowski decomposition of \(\delta'\) satisfies
  \(
    \mathsf{M}(\delta)|_{\delta'} \preceq \mathsf{M}(\delta')
  \).
\end{definition}

\begin{remark}
  Pairs \((P,\mathsf{M})\) are also partially ordered by refinement, that is, if \((P,\mathsf{M})\) and \((P,\mathsf{M}')\) are Minkowski decompositions of \(P\), we say that \((P,\mathsf{M}) \preceq (P,\mathsf{M}')\) if \(\mathsf{M}(\delta) \preceq \mathsf{M}'(\delta)\) for all \(\delta\). We let \(\mathsf{M}_{triv}\) denote the trivial Minkowski decomposition.
\end{remark}

If \(\mathsf{p}= \sum_{\rho \in P_\mathbb{Z}} c_\rho x^\rho\) is a Laurent polynomial with Newton polytope \(P\), then for each face \(\delta\) of \(P\) there is an associated \emph{face polynomial}
\[
  \mathsf{p}_\delta = \sum_{\rho \in \delta \cap P_\mathbb{Z}} c_\rho x^\rho.
\]

\begin{definition}
  Let \(\mathsf{M}\) be a Minkowski datum for a reflexive polytope \(P\). A Laurent polynomial \(\mathsf{p}\) with Newton polytope \(P\) is an \emph{\(\mathsf{M}\)-polynomial} if for each face \(\delta\) of \(P\) we have \(\mathsf{M}(\delta) = \{(\delta_1,n_1), \ldots, (\delta_k,n_k)\}\), and there are polynomials \(\mathsf{h}_{\delta_1},\dots, \mathsf{h}_{\delta_k}\) so that the Newton polytope of \(\mathsf{h}_{\delta_i}\) is \(\delta_i\), and
  \(
    \mathsf{p}_\delta = (\mathsf{h}_{\delta_1}^{n_1} \cdots \mathsf{h}_{\delta_k}^{n_k})x^\nu
  \)
  for some monomial \(x^\nu\).
\end{definition}

Definition of \(\mathsf{M}\)-polynomials is inspired by \emph{Minkowski polynomials} of~\cite{akhtar/minkowski}.

\begin{definition}[see~\cite{akhtar/minkowski}]\label{definition:minkowski}
  Suppose \(\mathsf{p}\) is a Laurent polynomial with reflexive Newton polytope of dimension \(3\). We say that \(\mathsf{p}\) is a \emph{Minkowski polynomial} if it is an \(\mathsf{M}\)-polynomial so that for each face \(\delta\) we have \(\mathsf{M}(\delta) = \sum_{i=1}^n n_i \delta_i\), where \(\delta_i\) is either affinely equivalent to an \(A_n\)-triangle with vertices \((0,0),(1,0)\) and \((0,n)\), or an interval of length \(1\). We require that if \(\delta_i\) is an \(A_n\)-triangle, then, after possible change of variables, we have \(\mathsf{p}_{\delta_i} = (1 + y)^n + x\). If \(\delta_i\) is an interval of length 1, we require that, after a possible change of variables, \(\mathsf{p}_{\delta_i} = (1 + x)\).
\end{definition}

Suppose we are given an \(\mathsf{M}\)-polynomial \(\mathsf{p}\) with reflexive Newton polytope \(P\). Let~\(P^{\vee}\) denote the dual polytope, and let \({X}_{P^{\vee}}\) denote a crepant resolution of the toric variety associated with the face fan of the polytope \(P^{\vee}\), if such a resolution exists. Such a resolution always exists if \(\dim(P) \leq 3\). In this case the vanishing locus of \(\mathsf{p}\) in \(X_{P^{\vee}}\) provides a possibly singular hypersurface, which we can denote \(F_{\mathsf{p}}\), determined by the vanishing of a section \(\sigma_{\mathsf{p}}\) of the anticanonical bundle of \(X_{P^{\vee}}\). Each cone \(c\) in the fan \(\Sigma\) underlying \(X_{P^{\vee}}\) corresponds to a toric stratum \(\mathbb{T}_c\) in \(X_{P^{\vee}}\). Let \(\rho_1, \ldots, \rho_k\) be ray generators of \(c\). The intersection \(\mathbb{T}_c \cap F_{\mathsf{p}}\) is the vanishing locus of \(\mathsf{p}_{c^*}\), where
\[
  c^* = \{ m \in P \mid \langle m, \rho_i \rangle = -1 \, \text{ for all } i = 1, \ldots, k\}.
\]

\begin{definition}\label{definition:transverse-nondeg}
  We say that an \(\mathsf{M}\)-polynomial \(\mathsf{p}\) is \emph{transverse non-degenerate} if for each proper face \(\delta\) of \(P\) and each subface \(\tau \subseteq \delta\) of positive dimension the following hold:
  \begin{itemize}
  \item the relative normal cone \(N_{\tau/\delta}\) is unimodular in the dual of the saturated lattice \(M \cap \operatorname{span}_{\mathbb{R}}(\delta-\delta)\), hence there exist normal coordinates \(z_1, \ldots, z_{\dim(\delta) - \dim(\tau)}\) and a toric chart of the form \(U_{\delta,\tau} = \mathbb{A}^{\dim(\delta) - \dim(\tau)} \times \mathbb{G}_m^{\dim(\tau)}\);
  \item the closure of \(V(\mathsf{p}_\delta)_{\mathrm{red}}\) in \(U_{\delta,\tau}\) is SNC along the stratum \(z_1 = \ldots z_{\dim(\delta) - \dim(\tau)} = 0\).
  \end{itemize}
\end{definition}

Suppose we are given a transverse non-degenerate \(\mathsf{M}\)-polynomial \(\mathsf{p}\) supported on a reflexive Newton polytope \(P\). We obtain a pencil of anticanonical hypersurfaces
\[
  \mathcal{P}_{\mathsf{p}}: s \sigma_{\mathsf{p}} + t \prod_{\rho \in \Sigma(1)} x_\rho = 0.
\]
Let \(D_P\) denote the snc union of all torus invariant divisors in \(X_{P^{\vee}}\). The base locus of this pencil is the intersection of \(\sigma_{\mathsf{p}}\) with \(D_P\). Given a toric stratum \(\mathbb{T}_c\) attached to a cone \(c\) of~\(\Sigma\), the vanishing locus of \(\sigma_\mathsf{p}\) in \(\mathbb{T}_c\) can be computed using \(P^{\vee}\) as discussed before.

\begin{proposition}[Przyjalkowski, see~{\cite[Proposition~3.6]{doran/modularity}}]\label{proposition:compactification}
  Let \(\mathsf{p}\) be a transverse non-degenerate Laurent polynomial in \(n\) variables with reflexive Newton polytope \(P\). Assume that the toric Fano variety attached to \(P^{\vee}\) admits a crepant resolution \(X_{P^{\vee}}\). There exists a tame compactification obtained by systematically resolving the base locus of \(\mathcal{P}_\mathsf{p}\), such that its fibres are compactifications of the fibres of \(\mathsf{p}\).
\end{proposition}

\begin{remark}\label{remark:compactification}
  This crepant resolution exists provided that \(P^{\vee}\) admits a fine regular star triangulation with associated smooth face fan (see~\cite{batyrev/polyhedra} or~\cite[\nopp 2.3]{doran/k-theory}). For \(\dim(P) = 4\) one could also use the existence of \emph{partial} crepant resolutions to construct a weak log-Calabi--Yau compactification (see~\cite[Remark~3.7]{doran/modularity}). Namely, for \(\dim(P) = 4\) a partial crepant resolution has only isolated torus-invariant singularities, they do not lie in the base locus of the pencil as the corresponding section does not vanish at such points. Consequently, these singular points are contained in the fibre over infinity of the obtained family.

  The original statement (i.e.,~\cite[Proposition~3.6]{doran/modularity}) imposes a much weaker condition on a Laurent polynomial yet its proof assumes that the reduced base locus divisor has simple normal crossings. To ensure this, we use transverse non-degeneracy condition.
\end{remark}

Here we also review some facts from Batyrev's approach to Minimal Model Program of non-degenerate toric hypersurfaces (see~\cite{batyrev/MMP}).

\begin{definition}
  A \(k\)-dimensional lattice polytope \(P' \subset \mathbb{R}^k\), \(k = 1, \ldots, d\), is called
\emph{lattice projection} of a \(d\)-dimensional lattice polytope \(P \subset \mathbb{R}^d\) if there exists an affine map \(\pi \colon \mathbb{R}^d \rightarrow \mathbb{R}^k\) inducing a surjective map of lattices \(\pi \colon \mathbb{Z}^d \rightarrow \mathbb{Z}^k\), and \(\pi(P) = P'\).
\end{definition}

\begin{definition}
  A \(d\)-dimensional lattice polytope \(P\) has \emph{width 1} if \(P\) has a lattice projection on the lattice segment \([0, 1] \subset \mathbb{R}\).
\end{definition}

\begin{lemma}
  A non-degenerate toric hypersurface \(Z \subset \mathbb{T}^d_K\) over a field \(K\) with Newton polytope of width 1 is rational over \(K\).
\end{lemma}

\begin{conjecture}
  Let \(P\) be a \(d\)-dimensional lattice polytope. Assume that for any field \(K\) and any toric hypersurface \(Z \subset \mathbb{T}^d_K\) with Newton polytope \(P\) is irreducible and birational to \(\mathbb{A}^{d-1}_K\) over \(K\). Then the polytope \(P\) is of width one.
\end{conjecture}

\begin{definition}
  For a non-zero lattice vector \(\textbf{a} = (a_1, \ldots, a_d) \in \mathbb{Z}^d \setminus \{0 \}\) consider the \emph{integral supporting hyperplane} of \(P\):
  \[
    H_{\textbf{a}, P} :=
    \left \{
      \textbf{x} = (x_1, \ldots, x_d) \in \mathbb{R}^d \, \mid \,
      \sum_{i=1}^d a_i x_i = \min_P(\textbf{a}) :=
      \min_{\textbf{y} \in P} \left( \sum_{i=1}^d a_i y_i \right)
    \right \}.
  \]
  If \(\gcd(a_1, \ldots, a_d)=1\), the \emph{integral distance} between a point \(\textbf{y} = (y_1, \ldots, y_d) \in P\) and integral supporting hyperplane \(H_{\textbf{a},P}\) is the following number:
  \[
    \dist_{\mathbb{Z}}(\textbf{y}, H_{\textbf{a}, P}) :=
    \sum_{i=1}^d a_i y_i \; - \;
    \min_P(\textbf{a}) \geqslant 0.
  \]
\end{definition}

\begin{definition}
  Let \(P \subset \mathbb{R}^d\) be an arbitrary \(d\)-dimensional convex polytope. The set \(F(P)\) of all points in \(P\) having integral distance at least \(1\) to any integral supporting hyperplane \(H_{\textbf{a}, P}\) is called the \emph{Fine interior} of \(P\), i.e.,
  \[
    F(P) := \{ \textbf{y} \in \mathbb{R}^d \, \mid \, \dist_{\mathbb{Z}} (\textbf{y}, H_{\textbf{a}, P}) \geqslant 1, \;\; \forall \textbf{a} \in {\mathbb{Z}}^d \setminus \{0\} \}.
  \]
  If \(F(P)\) is empty, then the polytope \(P\) is called \emph{F-hollow}.
\end{definition}

\begin{theorem}
  A non-degenerate toric hypersurface \(Z \subset \mathbb{T}^d\) with Newton polytope~\(P\) has negative Kodaira dimension if and only if \(P\) is F-hollow.
\end{theorem}

\begin{corollary}[see~\cite{batyrev/MMP,kollar/surfaces}]\label{corollary:toric-surfaces}
  Let \(Z \subset \mathbb{T}^3_K\) be a non-degenerate toric hypersurface with a Newton polytope of width 1, where \(K = \mathbb{R}\) or \(K = \overline{\mathbb{Q}} \cap \mathbb{R}\). Then the surface \(Z\) is birational over \(K\) to \(\mathbb{P}^2\), \(\mathbb{P}^1 \times \mathbb{P}^1\) or a Hirzebruch surface.
\end{corollary}

\subsection{Properties of higher Chow groups}\label{subsection:chow}

In this subsection we review some properties of higher Chow groups. We refer to~\cite{elbaz-vincent/higher-chow} for a short introduction.

\begin{theorem}
  If \(X\) is a quasi-projective variety over a field, \(D \subset X\) is a closed subvariety of pure codimension \(r\), and \(Y = X \setminus D\). Then there exists the exact \emph{localisation sequence} of higher Chow groups:
  \[
    \ldots \rightarrow \CH^l(D, n) \rightarrow \CH^{l + r}(X, n) \rightarrow \CH^{l + r}(Y, n) \rightarrow
    \CH^l(D, n - 1) \rightarrow \ldots
  \]
\end{theorem}

\begin{corollary}
  Let \(K\) be any field. Then we have
  \[
    \CH^q(\mathbb{P}^1_K, n) \simeq \CH^q(\Spec(K), n) \oplus \CH^{q - 1}(\Spec(K), n).
  \]
\end{corollary}

From~\cite[Theorem~9.9]{burgos-gil/regulators} and~\cite[Theorem~5.4.3]{elbaz-vincent/higher-chow} we obtain

\begin{theorem}
  Let \(K\) be a number field. Then we have \(\CH^q(\Spec(K), 2q - p)_{\mathbb{Q}} = 0\) unless \(p = q = 0\) or \(p = 1\). For \(p = 1\) and \([K \colon \mathbb{Q}] = r_1 + 2 r_2\), we have
    \[
      \CH^q(\Spec(K), 2 q - 1)_{\mathbb{Q}} \simeq
      \begin{cases}
        K^*_{\mathbb{Q}}, & q = 1, \\
        \mathbb{Q}^{r_2}, & q \geqslant 2 \text{ even;} \\
        \mathbb{Q}^{r_1 + r_2}, & q \geqslant 3 \text{ odd.}
      \end{cases}
    \]
\end{theorem}

\begin{proposition}[see~{\cite[Example~5.1.8]{elbaz-vincent/higher-chow}}]\label{proposition:base-change}
  Let \(L \supset K\) be a finite field extension, and \(X\) be a quasi-projective variety over \(K\). Then \(\CH^q(X, n)_{\mathbb{Q}} \hookrightarrow \CH^q(X_L, n)_{\mathbb{Q}}\).
\end{proposition}

From this, the localisation sequence, and~\cite[Corollary~5.1.11]{elbaz-vincent/higher-chow} we obtain

\begin{corollary}\label{corollary:removing-points}
  Let \(K\) be a number field, \(X\) be a quasi-projective variety over \(K\), \(D \subset X\) be a subvariety of codimension \(m = \dim(X)\), and \(Y = X \setminus D\). Then
  \[
    \CH^{l + m}(Y, n)_{\mathbb{Q}} \simeq \CH^{l + m}(X, n)_{\mathbb{Q}}, \quad
    n \neq 2 l - 1, 2 l, \quad l > 0.
  \]
\end{corollary}

\begin{corollary}\label{corollary:affine-vanishing}
  Let \(K\) be a number field, \(X = \mathbb{A}^m_K\), \(D \subset X\) be a (totally reducible) subvariety of codimension \(m = \dim(X)\), and \(Y = X \setminus D\). Then we have \(\CH^{l + m}(Y, n)_{\mathbb{Q}} = 0\) for \(n \neq 2 l, 2 (l + m) - 1\). Now assume also that \(K\) is totally real: if \(n = 2 (l + m) - 1\), and \(l + m\) is even, or \(n = 2 l\), and \(l\) is even, then we also have \(\CH^{l + m}(Y, n)_{\mathbb{Q}} = 0\).
\end{corollary}

\begin{corollary}\label{corollary:projective-vanishing}
  Let \(K\) be a number field. If \(n \neq 2 q - 1, 2 q - 3\), \(q > 1\), then \(\CH^q(\mathbb{P}^1_K, n)_{\mathbb{Q}} = 0\). Now assume also that \(K\) is totally real: if \(n = 2 q - 1\), and \(q\) is even, or \(n = 2 q - 3\), and \(q\) is odd, then we also have \(\CH^q(\mathbb{P}^1_K, n)_{\mathbb{Q}} = 0\).
\end{corollary}

\begin{theorem}[{\cite[Corollary~5.4]{levine/bloch}}]\label{theorem:projective-bundle}
  Let \(E \rightarrow X\) be a vector bundle of rank \(r + 1\) over a smooth quasi-projective variety \(X\), and let \(\pi \colon \mathbb{P}(E) \rightarrow X\) be the associated projective space bundle. Let \(\zeta\) be the class of \(\mathcal{O}(1)\) in \(\CH^1(\mathbb{P}(E))\). Then the maps
  \[
    \alpha_i \colon \CH^{q - i}(X, \bullet)_{\mathbb{Q}} \rightarrow \CH^q(\mathbb{P}(E), \bullet)_{\mathbb{Q}}, \quad
    \alpha_i(\eta) = \pi^*(\eta) \cup \zeta^i, \quad i = 0, \ldots, r,
  \]
  define an isomorphism for each \(p\) and each \(q\):
  \[
    \sum_{i = 0}^r \alpha_i \colon \bigoplus_{i = 0}^r
    \CH^{q - i}(X, p)_{\mathbb{Q}} \xrightarrow{\sim} \CH^q(\mathbb{P}(E), p)_{\mathbb{Q}}.
  \]
\end{theorem}

\begin{corollary}\label{corollary:hirzebruch-vanishing}
  Let \(K\) be a totally real number field, and \(X\) be a Hirzebruch surface defined over \(K\). Then we have \(\CH^3(X, 3)_{\mathbb{Q}} = 0\).
\end{corollary}

\begin{theorem}[{\cite[Theorem~9.1]{mazza/motivic}}]
  Let \(X\) be a smooth separated scheme over a perfect field. Then for all \(n\) and \(i \geqslant 0\) there is a natural isomorphism \(\CH^i(X, 2i - n) \simeq H^n_{\mathcal{M}}(X, \mathbb{Z}(i))\), where \(H^n_{\mathcal{M}}(X, \mathbb{Z}(i))\) are motivic cohomology groups.
\end{theorem}

\section{Proof of Theorem~\ref{theorem:amenable}}\label{section:proof-amenable}

Note that Minkowski polynomials are precisely transverse non-degenerate Laurent polynomials \(\mathsf{p} \in \mathbb{Z}[x_1^{\pm}, x_2^{\pm}, x_3^{\pm}]\) satisfying the binomial principle whose facets of Newton polytope admit lattice Minkowski decomposition into segments or triangles of width 1. Based on this, we propose a generalisation of Minkowski polynomials.

\begin{definition}
  Let \(P\) be a reflexive polytope of \(\dim(P) \leqslant 4\). We say that \(P\) is
  \begin{itemize}
  \item \emph{amenable} if any face of \(P\) admits a lattice Minkowski decomposition into lattice polytopes of width one (so their faces are of this form too);
  \item \emph{strictly amenable} if \(P\) is amenable, and \(P^{\vee}\) admits a fine regular star triangulation with smooth face fan (see~\cite{batyrev/polyhedra} or~\cite[\nopp 2.3]{doran/k-theory}).
  \end{itemize}
\end{definition}

\begin{remark}
  The last assumption is trivial for \(\dim(P) < 4\), this follows from existence of maximal projective crepant partial desingularisation (see~\cite{batyrev/polyhedra} or~\cite[\nopp 2.3]{doran/k-theory}).
\end{remark}

\begin{definition}\label{definition:totally-rational}
  Let \(\mathsf{q} \in \mathbb{C}[x_1^{\pm}, \ldots, x_m^{\pm}]\), \(m \leqslant 3\), be a Laurent polynomial, and \(K \subset \mathbb{C}\) be its field of definition. We say that \(\mathsf{q}\) is \emph{totally rational} if \(\mathsf{q} = \mathsf{q}_1 \cdot \cdots \cdot \mathsf{q}_r\) is totally reducible over \(K\), and for each component \(\mathsf{q}_j\) the following assumptions hold:
  \begin{itemize}
  \item Newton polytope of \(\mathsf{q}_j\) is of lattice width one, i.e., after a unimodular change of coordinates it is of the form \(\mathsf{q}_j = F_j(x_1, \ldots, x_{m - 1}) x_m + G_j(x_1, \ldots, x_{m - 1})\);
  \item the intersections \(\cap_{j \in J} \{\mathsf{q}_j = 0\}\) are rational over \(K\) for any \(J \subset \{1, \ldots, r\}\);
  \item \(F_j(x_1, \ldots, x_{m - 1}) = G_j(x_1, \ldots, x_{m - 1}) = 0\) is totally reducible over \(\mathbb{R} \cap \overline{\mathbb{Q}}\).
  \end{itemize}
\end{definition}

\begin{remark}
  In practice totally rational Laurent polynomials arise as face polynomials of some Laurent polynomial whose edge polynomials are of the form \((1 + t)^l\) (cf. Definition~\ref{definition:minkowski}). In other words, its coefficients are uniquely determined, and ``total rationality'' is a property of Minkowski summands of facets of Newton polytope.
  
  Actually, the third assumption in Definition~\ref{definition:totally-rational} would be used only for \(m = 3\). In this case an irreducible totally rational Laurent polynomial \(\mathsf{q}_j \in \mathbb{C}[x_1^{\pm}, x_2^{\pm}, x_3^{\pm}]\) defines a toric hypersurface \(\{\mathsf{q}_j = 0\} \subset \mathbb{G}_m^3\) which admits a decomposition
  \[
    \{\mathsf{q}_j = 0\} = \{(x_1, x_2) \in \mathbb{G}_m^2 : (F_j \cdot G_j) (x_1, x_2) \neq 0\} \sqcup \{(x_1, x_2) \in \mathbb{G}_m^2 : F_j = G_j = 0\} \times \mathbb{G}_m.
  \]
\end{remark}

\begin{definition}
  Let \(\mathsf{p} \in \mathbb{C}[x_1^{\pm}, \ldots, x_n^{\pm}]\), \(n \leqslant 4\), be a transverse non-degenerate Laurent polynomial with an amenable Newton polytope. We assume that Minkowski decompositions of faces are compatible with factorisations of face polynomials. We say that \(\mathsf{p}\) is
  \begin{itemize}
  \item \emph{weakly amenable} if edge polynomials vanish only at roots of unity;
  \item \emph{amenable} (respectively, \emph{strictly amenable}) if edge polynomials are products of cyclotomic polynomials, and the Newton polytope is (strictly) amenable. For \(n = 4\) we require that face polynomials are \emph{totally rational} (see Definition~\ref{definition:totally-rational}).
  \end{itemize}
\end{definition}

\begin{remark}\label{remark:rationality}
  For \(\dim(P) < 4\) amenability and strict amenability are the same thing. A non-degenerate Laurent polynomial has Newton polytope of lattice width one if and only if it could be written as \(P(y_1, \ldots, y_{m - 1}) y_m + Q(y_1, \ldots, y_{m - 1}) = 0\) for some Laurent polynomials \(P\) and \(Q\), so \(K\)-rationality of irreducible components of zero locus of any face polynomial is immediate. Intersections of irreducible components for different face polynomials are also of this form. The only problematic point are intersections of irreducible components for the same face polynomial. In examples from Proposition~\ref{proposition:examples} and Theorem~\ref{theorem:apery} the \(K\)-rationality could be checked by direct computations. In other words, for our purposes transverse non-degenerate Laurent polynomials behave closely enough to the classical non-degenerate Laurent polynomials of Danilov--Khovanskii, but one also needs to keep track of intersections of components of zero loci of face polynomials.
\end{remark}

Introduced Laurent polynomials are very close to \emph{tempered} Laurent polynomials in the sense of Villegas and Doran--Kerr (see~\cite{villegas/mahler,doran/k-theory,kerr/motivic,kerr/unipotent}).

\begin{definition}
  Let \(\mathsf{p} \in \mathbb{C}[x_1^{\pm}, \ldots, x_n^{\pm}]\) be a Laurent polynomial with reflexive Newton polytope \(\Delta\) admitting a weak log Calabi--Yau compactification (see Definition~\ref{definition:log-CY}). We say that the Laurent polynomial \(\mathsf{p}\) is \emph{tempered} if the toric-coordinate symbol \(\{x_1, \ldots, x_n\}\) in \(\CH^{n}(\mathbb{G}_m^n, n)_{\mathbb{Q}}\) completes to a motivic cohomology class in \(H^{n}_{\mathcal{M}}(\mathcal{W} \setminus \mathcal{W}_0, \mathbb{Q}(n))\), where \(\mathcal{W}_0\) is the fibre of the potential \(\mathsf{f}\) over infinity.
\end{definition}

\begin{remark}[see~\cite{villegas/mahler}]
  For \(n = 2\) a Laurent polynomial with reflexive Newton polygon is tempered if and only if edge polynomials vanish only at roots of unity. In other words, it is precisely a weakly amenable Laurent polynomial.
\end{remark}

\begin{proposition}[see~\cite{daSilva/arithmetic}]
  Minkowski polynomials are tempered.
\end{proposition}

\begin{notation}\label{notation:toric}
  Let \(\mathsf{p} \in \mathbb{C}[x_1^{\pm}, \ldots, x_n^{\pm}]\) be a Laurent polynomial with reflexive Newton polytope \(\Delta\) satisfying assumptions of Proposition~\ref{proposition:compactification} or Remark~\ref{remark:compactification}, and \(\mathbb{P}_{tr(\Delta^{\vee})^{\vee}} \rightarrow \mathbb{P}_{\Delta}\) be a (partial) toric crepant resolution. Let us denote by \(\{X^{\lambda} = X_t\}\), \(t = \lambda^{-1}\), the pencil of anticanonical hypersurfaces
  \(
    \mathbb{P}^1 \times \mathbb{P}_{\Delta} \supset \mathcal{X} \rightarrow \mathbb{P}^1
  \)
  given by taking the Zariski closure of \(\{1 - t \mathsf{p} = 0\} = \{\mathsf{p} = \lambda\} \subset \mathbb{A}^1 \times \mathbb{G}_m^n\). Its base locus is \(X^{\lambda} \cap X_0\), where \(X_0 = \mathbb{D} \subset \mathbb{P}_{\Delta}\), and \(\mathbb{D}\) is the union of toric boundary components:
  \[
    \mathbb{D} = \bigcup_{\delta \in \Delta(1)} \mathbb{D}_{\delta} \subset
    \bigsqcup_{i = 1}^n \left (
      \bigsqcup_{\delta \in \Delta(i)} \mathbb{D}^*_{\delta} \right ) = \mathbb{P}_{\Delta} \setminus \mathbb{G}_m^n, \quad
    \mathbb{D}^*_{\delta} \cong \mathbb{G}_m^{n - i}.
  \]
  Moreover, \(D_{\delta} = \mathbb{D}_{\delta} \cap X_t\), \(D_{\delta}^* = \mathbb{D}_{\delta}^* \cap X_t\), \(D = \mathbb{D} \cap X_t\) are independent of \(t \neq 0\), and
  \[
    \widetilde{\mathcal{X}}_{-} = \widetilde{\mathcal{X}} \setminus (\{\infty\} \times \widetilde{X}^{\infty}) \subset \mathbb{A}^1_{\lambda} \times \mathbb{P}_{tr(\Delta^{\vee})^{\vee}}, \quad
    \widetilde{X} = \{(\lambda, x) \; \vert \; x \in \widetilde{\mathcal{X}}^{\lambda}\} \subset \mathbb{P}^1_{\lambda} \times \mathbb{P}_{tr(\Delta^{\vee})^{\vee}}.
  \]
  We also have \(\widetilde{X}_0 = \widetilde{X}^{\infty} = \widetilde{\mathbb{D}}\), and \(\widetilde{X} \cap (\mathbb{A}^1 \times \widetilde{\mathbb{D}}) \simeq \mathbb{A}^1 \times \widetilde{D}\), where \(\widetilde{D} = \widetilde{\mathbb{D}} \cap \widetilde{X}_t\), and
  \[
    \widetilde{\mathbb{D}} = \bigcup_{\widetilde{\delta} \in tr(\Delta^{\vee})^{\vee}(1)} \mathbb{D}_{\widetilde{\delta}} =
    \bigsqcup_{i = 1}^n \left (
      \bigsqcup_{\widetilde{\delta} \in tr(\Delta^{\vee})^{\vee}(i)} \mathbb{D}^*_{\widetilde{\delta}} \right ) =
    \mathbb{P}_{tr(\Delta^{\vee})^{\vee}} \setminus \mathbb{G}_m^n, \quad \mathbb{D}_{\widetilde{\delta}} \twoheadrightarrow \mathbb{D}_{\delta} \subset \mathbb{P}_{\Delta}.
  \]
  Here we denote by \(\widetilde{X}_t\) the strict transform of \(X_t\) (see~\cite[Section~2]{doran/k-theory} for details).
\end{notation}

\begin{remark}
  If \(D^* \subset \mathbb{D}^* \cong \mathbb{G}_m^{n - i}\) is an irreducible smooth toric hypersurface, then the choice of coordinate functions \(x_1, \ldots, x_{n - i} \in \Gamma(\mathbb{D}^*, \mathcal{O}^*_{\mathbb{D}^*})\) defines a class \(\{x_1, \ldots, x_{n - i}\}\) in \(\CH^{n - i}(\mathbb{D}^*, n - i)_{\mathbb{Q}}\) which in turn restricts to \(\CH^{n - i}(D^*, n - i)_{\mathbb{Q}}\) (see~\cite[\nopp 3]{doran/k-theory}).
\end{remark}

\begin{proposition}\label{proposition:amenable}
  Let \(\mathsf{p} \colon \mathbb{G}_m^n \rightarrow \mathbb{C}\) be a Laurent polynomial which is weakly amenable for \(n = 2,3\). For \(n = 4\) assume that it is amenable (respectively, strictly amenable), and that \(\mathsf{p}\) is defined over a totally real abelian number field. Then there exists a weak (respectively, smooth) log Calabi-Yau compactification \(\mathsf{f} \colon \mathcal{W} \rightarrow \mathbb{P}^1\), and \(\mathsf{p}\) is tempered.
\end{proposition}

\begin{proof}
  Let \(\Delta\) be the reflexive Newton polytope, and \(\Delta^{\vee}\) be its dual. Then there exists a fine regular star triangulation of \(\Delta^{\vee}\) defining a toric variety \(\mathbb{P}_{tr(\Delta^{\vee})^{\vee}}\) (defined by the subdivision of the face fan of \(\Delta^{\vee}\), i.e., of the normal fan of \(\Delta\)). It is known that such triangulations correspond to maximal projective crepant partial desingularisations \(\mathbb{P}_{tr(\Delta^{\vee})^{\vee}} \rightarrow \mathbb{P}_{\Delta}\) (see~\cite{batyrev/polyhedra} or~\cite[\nopp 2.3]{doran/k-theory}). Moreover, it is known that \(\Sing(\mathbb{P}_{tr(\Delta^{\vee})^{\vee}})\) is of codimension at least 4 (i.e., \(\mathbb{P}_{tr(\Delta^{\vee})^{\vee}}\) has at worst isolated singularities), and that singularities are at worst \(\mathbb{Q}\)-factorial terminal. If \(\Delta\) is strictly amenable, then \(\mathbb{P}_{tr(\Delta^{\vee})^{\vee}}\) is smooth. So by Proposition~\ref{proposition:compactification} and Remark~\ref{remark:compactification} there exists a (weak) log-Calabi--Yau compactification \(\mathcal{W} \rightarrow \mathbb{P}^1\) decomposing as \(\mathcal{W} \xrightarrow{\mathrm{Bl}} \widetilde{\mathcal{X}} \dashrightarrow \mathbb{P}^1\) (we follow Notation~\ref{notation:toric}). Then it is enough to construct the required class in \(\CH^{n}(\widetilde{\mathcal{X}}_{-}, n)_{\mathbb{Q}}\) and take its pull-back to \(\mathcal{W} \setminus \mathcal{W}_0\) (which is correctly defined and is injective, see~\cite[Theorem~1.2]{hanamura/blowups}). By assumption \(\mathsf{p}\) up to shift to a constant is defined over some subfield \(K\) of a cyclotomic field (equivalently, \(K\) is an abelian number field). If \(n = 4\), we also require that \(K\) should be totally real.

  Note that the variety \(\widetilde{\mathcal{X}}_{-} \setminus (\mathbb{A}^1 \times \widetilde{D})\) is actually isomorphic to \(\mathbb{G}_m^n\), i.e., \(\widetilde{\mathcal{X}}_{-} = \mathrm{Bl}(\mathcal{W} \setminus \mathcal{W}_0)\) is a fibre-wise compactification of the torus over \(\mathbb{A}^1\) by the SNC divisor \(B = \mathbb{A}^1 \times \widetilde{D}\). Even if our Laurent polynomial is only amenable and not strictly amenable, \(\widetilde{\mathcal{X}}_{-}\) is still smooth (see Remark~\ref{remark:compactification}). Then we could check that the symbol \(\{x_1, \ldots, x_n\} \in \CH^{n}(\widetilde{\mathcal{X}}_{-} \setminus B, n)_{\mathbb{Q}}\) completes to a motivic cohomology class in \(H^{n}_{\mathcal{M}}(\widetilde{\mathcal{X}}_{-}, \mathbb{Q}(n))\) by applying the \emph{coniveau spectral sequence} (see~\cite[\nopp 3.4]{kerr/abel-jacobi}). Let \(B = \cup_{i = 1}^N B_i\) be the decomposition into irreducible components, and put \(B^k = \cup_{\vert I \vert = k} \cap_{i \in I} B_i\) and \(\widetilde{B^k} = \sqcup_{\vert I \vert = k} \cap_{i \in I} B_i\). Then
  \begin{gather*}
    N^k C^{2p + *}_{\mathcal{M}}(\widetilde{\mathcal{X}}_{-}, \mathbb{Q}(p)) =
    \im \{ (2 \pi i)^k i_*^{\widetilde{B^k}} : C_{\mathcal{M}}^{2p - 2k + *} (\widetilde{B^k}, \mathbb{Q}(p - k)) \rightarrow
    C^{2 p + *}_{\mathcal{M}}(\widetilde{\mathcal{X}}_{-}, \mathbb{Q}(p))\}, \\
    \mathcal{M}(p)^{a,b}_0 = Gr^a_{-N} C_{\mathcal{M}}^{2p + a + b} (\widetilde{\mathcal{X}}_{-}, \mathbb{Q}(p)), \quad
    C^{*}_{\mathcal{M}}(\widetilde{\mathcal{X}}_{-}, \mathbb{Q}(p)) := Z^p_{\mathbb{R}}(\widetilde{\mathcal{X}}_{-}, 2p - *)
  \end{gather*}
  (see~\cite{kerr/abel-jacobi} for notation) define a fourth-quadrant spectral sequence with
  \[
    \mathcal{M}_1^{a,b} \simeq H^{2p - a + b}_{\mathcal{M}} (B^a \setminus B^{a + 1}, \mathbb{Q}(p - a)), \quad
    \mathcal{M}_{\infty}^{a,b} = Gr^a_{-N} H_{\mathcal{M}}^{2p + a + b}(\widetilde{\mathcal{X}}_{-}, \mathbb{Q}(p)).
  \]
  The higher differentials \(d_r\) on \(\mathcal{M}(p)^{0,b}_r\) yield \emph{higher residue maps} (see~\cite[\nopp 3.1]{kerr/abel-jacobi})
  \begin{gather*}
    Res^r \colon \{ \ker(Res^{r - 1}) \subset H_{\mathcal{M}}^{2p + b} (\widetilde{\mathcal{X}}_{-} \setminus B, \mathbb{Q}(p)) \} \rightarrow
    \\ \rightarrow \{\text{subquotient of } H^{2p + b - 2r + 1}_{\mathcal{M}} (B^r \setminus B^{r + 1}, \mathbb{Q}(p - r))\}
  \end{gather*}
  with \(\cap_{r = 1}^{\lfloor2 p + b - p/2\rfloor} \ker(Res^r) = \im (H^{2p + b}_{\mathcal{M}}(\widetilde{\mathcal{X}}_{-}, \mathbb{Q}(p)) \rightarrow H^{2p + b}_{\mathcal{M}}(\widetilde{\mathcal{X}}_{-} \setminus B, \mathbb{Q}(p)))\).

  In our situation \(p = n\), \(b = -n\), and actually we have \(r = 1\) for \(n = 2,3\), and \(r = 1, 2\) for \(n = 4\). Recall that \(\widetilde{D} = \cup_{\widetilde{\delta} \in tr(\Delta^{\vee})^{\vee}(1)} D_{\widetilde{\delta}}\), where we have \(D_{\widetilde{\delta}}^* \simeq \mathbb{G}_m^d \times D_{\delta}^*\) for some \(d\). For \(n = 4\) our assumptions imply that \(Res^2\) vanishes: the target of \(Res^2\) is a subquotient of \(H^1_{\mathcal{M}} (B^2 \setminus B^3, \mathbb{Q}(2)) \simeq \CH^2(B^2 \setminus B^3, 3)_{\mathbb{Q}}\). Note that \(B^2 \setminus B^3\) is the product of \(\mathbb{A}^1\) and a disjoint union of open subsets of intersection of surfaces. We could ignore the \(\mathbb{A}^1\) multiple, and by our assumptions these open curves are rational over \(K\) (see Remark~\ref{remark:rationality}), so could be realised as open subsets of \(\mathbb{A}^1_K\). It is not hard to check that \(\CH^2(W, 3)_{\mathbb{Q}} = 0\) for open subsets \(W \subset \mathbb{A}^1\) over a totally real number field (see Subsection~\ref{subsection:chow}), so \(Res^2\) vanishes.

  Now we have to check that the symbol \(\{x_1, \ldots, x_n\}\) in \(H^n_{\mathcal{M}}(\widetilde{\mathcal{X}}_{-} \setminus B, \mathbb{Q}(n))\) lies in \(\ker(Res^1)\). We work with \(B^1 \setminus B^2\), which is the union of irreducible components of the boundary divisor minus their intersections. Recall that \(B = B^1 = \mathbb{A}^1 \times \widetilde{D}\), where \(\widetilde{D} = \bigcup_{\widetilde{\delta} \in tr(\Delta^{\vee})^{\vee}(1)} D_{\widetilde{\delta}}\). Let \(D_{\widetilde{\delta}} = \cup_l D_{\widetilde{\delta},l}\) be the decomposition into irreducible components over \(K\), and \(D_{\widetilde{\delta},l}^* = D_{\widetilde{\delta},l} \cap D^*_{\widetilde{\delta}} = D_{\widetilde{\delta},l} \setminus \cup_{\widetilde{\delta}' \subset \widetilde{\delta}} (D_{\widetilde{\delta},l} \cap D_{\widetilde{\delta}'})\). Note that the subset \(\cup_{\widetilde{\delta}' \subset \widetilde{\delta}} (D_{\widetilde{\delta},l} \cap D_{\widetilde{\delta}'}) \subset D_{\widetilde{\delta},l}\) is closed, and we have \(D_{\widetilde{\delta}}^* = \cup_l D_{\widetilde{\delta},l}^*\). Then we have \(B^1 \setminus B^2 \subset \cup D^*_{\widetilde{\delta}, l}\). Strictly speaking, \(B^1 \setminus B^2\) is a \emph{disjoint} union, while \(\cup D^*_{\widetilde{\delta}, l}\) is not: we also have to cut out intersections of the form \(W = D^*_{\widetilde{\delta}, l_1} \cap D^*_{\widetilde{\delta}, l_2}\). Fortunately, they are rational over \(K\) (see Remark~\ref{remark:rationality}), so \(\CH^{n - 2}(W, n - 1)_{\mathbb{Q}} = 0\), and it is enough to check that the symbol vanishes in \(\CH^{n - 1}(D^*_{\widetilde{\delta},l}, n - 1)_{\mathbb{Q}}\).

  We could check that the toric-coordinate symbol is zero in \(\CH^{n - i}(D^*_{\delta,l}, n - i)_{\mathbb{Q}}\) for all \(i \geqslant 1\), \(\delta \in \Delta(i)\), and all \(l\) instead of \(\CH^{n - 1}(D^*_{\widetilde{\delta},l}, n - 1)_{\mathbb{Q}}\) by applying~\cite[Proposition~3.5]{doran/k-theory}. Here we follow a similar notation for \(D_{\delta}^*\). So we only have to prove that \(\{x_1^{\delta}, \ldots, x_{n - i}^\delta\}\) give trivial classes in \(\CH^{n - i}(D^*_{\delta,l}, n - i)_{\mathbb{Q}}\) for all \(i \geqslant 1\) and \(\delta \in \Delta(i)\). We are going to do it by induction on \(S = \dim(D^*_{\delta,l}) = n - i - 1\).

  Note that every \(D_{\delta,l}\) is smooth, projective and rational over \(K\): they are smooth because \(\mathsf{p}\) is a transverse non-degenerate Laurent polynomial (see Subsection~\ref{subsection:toric}). We require that faces of Newton polytope admit lattice Minkowski decompositions into polytopes of width one, so irreducible components \(D_{\delta,l}\) are rational over \(K\) (see Remark~\ref{remark:rationality}).

  \begin{itemize}
  \item \(S = 0\): any \(D_{\delta,l} = D_{\delta,l}^*\) is a point, and the symbol \(\{x_1^{\delta}\}\) vanishes if and only if the edge polynomial vanishes only at roots of unity (see~\cite[Remark~3.3(b)]{doran/k-theory}); for \(n = 2\) we recover the usual definition of temperedness.
  \item \(S = 1\): any \(D_{\delta,l}\) is a smooth rational curve defined over \(K\). Let us apply the localisation sequence for the pair \((D_{\delta,l}, \cup_{\delta' \subset \delta} (D_{\delta,l} \cap D_{\delta'}))\). We obtain that
    \[
      \im \left (\CH^2(D_{\delta,l}, 2) \right) =
      \ker \left ( \CH^2(D_{\delta,l}^*, 2) \rightarrow \bigoplus_{\delta' \in \delta(1)} \CH^1(D_{\delta,l} \cap D_{\delta'}, 1) \right ).
    \]
    Note that \(\{x_1^{\delta}, x_2^\delta\}\) lies in RHS from the induction basis. At the same time \(D_{\delta,l} \simeq \mathbb{P}^1\) implies that \(\CH^2(D_{\delta,l}, 2)_{\mathbb{Q}} = 0\) (see Corollary~\ref{corollary:projective-vanishing}), hence the toric-coordinate symbol \(\{x_1^{\delta}, x_2^\delta\}\) is trivial in \(\CH^2(D^*_{\delta,l}, 2)_{\mathbb{Q}}\).
  \item \(S = 2\): any \(D_{\delta,l}\) is a smooth rational surface defined over \(K\). We also require that \(K\) is totally real. Then the surface \(D_{\delta,l}\) is isomorphic (over \(K' = \overline{\mathbb{Q}} \cap \mathbb{R}\)) to a blow up of \(\mathbb{P}^2\), \(\mathbb{P}^1 \times \mathbb{P}^1\) or a Hirzebruch surface in a finite number of real points or pairs of conjugate complex points by Corollary~\ref{corollary:toric-surfaces}.

    The localisation sequence for the pair \((D_{\delta,l}, \cup_{\delta' \subset \delta} (D_{\delta,l} \cap D_{\delta'}))\) implies that
    \[
      \im \left (\CH^3(D_{\delta,l}, 3) \right) =
      \ker \left ( \CH^3(D_{\delta,l}^*, 3) \rightarrow
        \CH^2 \left (\bigcup_{\delta' \subset \delta} D_{\delta,l} \cap D_{\delta'}, 2 \right ) \right ).
    \]
    Each irreducible component \(W \subset D_{\delta,l} \cap D_{\delta'}\) is isomorphic to \(\mathbb{P}^1\), hence we have \(\CH^2 \left (W, 2 \right )_{\mathbb{Q}} = 0\) (see Subsection~\ref{subsection:chow}). We could apply the local-global spectral sequence (see~\cite[\nopp 1.1]{doran/k-theory}) to conclude that the symbol \(\{x_1^{\delta}, x_2^\delta, x_3^\delta\}\) lies in the right hand side. Note that this is actually possible by previous steps of the induction. The only thing left is to show that \(\CH^3(D_{\delta,l}, 3)_{\mathbb{Q}} = 0\).

    Let \(\pi \colon D_{\delta,l} \rightarrow S\) be the blowing down to a minimal rational surface, i.e., let the surface \(S\) be either \(\mathbb{P}^2\), \(\mathbb{P}^1 \times \mathbb{P}^1\) or a Hirzebruch surface. Denote by \(E\) the exceptional divisor, which is a union of smooth rational curves defined over \(K'\) (and pairs of conjugated ones). Put \(U = D_{\delta,l} \setminus E\). From the localisation sequence for \((D_{\delta,l}, E)\) we obtain that \(\CH^3(D_{\delta,l}, 3)_{\mathbb{Q}} = 0\) provided that \(\CH^3(U, 3)_{\mathbb{Q}} = 0\) and \(\CH^2(E, 3)_{\mathbb{Q}} = 0\). Note that \(U\) is isomorphic to an open subset \(\pi(U) \subset S\) obtained by removing some points from \(S\). Note that by our assumptions components of \(E\) are defined over a totally real number field, so these points are real, though this does not matter yet. Corollary~\ref{corollary:removing-points} reduces this to \(\CH^3(S, 3)_{\mathbb{Q}} = 0\), which follows from Corollary~\ref{corollary:hirzebruch-vanishing}.

    Let us check that \(\CH^2(E, 3)_{\mathbb{Q}} = 0\). Consider the localisation sequence for \((E, N)\), where \(N\) is a union of intersections of components of \(E\). We have \(\CH^2(E, 3)_{\mathbb{Q}} = 0\) provided that \(\CH^2(E \setminus N, 3)_{\mathbb{Q}} = 0\) and \(\CH^1(N, 3)_{\mathbb{Q}} = 0\). The latter holds by Proposition~\ref{proposition:base-change}. Note that \(E \setminus N\) is a disjoint union of open subsets \(W\) of rational curves defined over a totally real number field, so \(\CH^2(W, 3)_{\mathbb{Q}} = 0\) (see Subsection~\ref{subsection:chow}). We obtain that \(\CH^3(D_{\delta,l}, 3)_{\mathbb{Q}} = 0\). \qedhere
  \end{itemize}
\end{proof}

\begin{remark}
  It's not clear how to generalise Proposition~\ref{proposition:amenable} to higher dimensions. One could ask for a combinatorial criterion on Newton polytopes of face polynomials such that corresponding non-degenerate toric hypersurfaces would compactify to smooth \(n\)-dimensional varieties \(X\) over some field \(L\) with \(\CH^{n + 1}(X, n + 1)_{\mathbb{Q}} = 0\).

  Nonetheless, Corollary~\ref{corollary:affine-vanishing} and Theorem~\ref{theorem:projective-bundle} imply that even for Bott towers we would eventually need to impose \(L = \mathbb{Q}\). In other words, this would require some delicate version of Batyrev's toric approach to Minimal Model Program in~\cite{batyrev/MMP}. Another problem is possible appearance of varieties with terminal singularities. Finally, in the coniveau spectral sequence we would have to consider higher residue maps as well.
\end{remark}

\section{Proof of Theorem~\ref{theorem:apery} and Corollary~\ref{corollary:apery}}\label{section:proof-apery}

The proof heavily depends on explicit calculations with Laurent polynomials requiring computer assistance, available at~\cite{ovcharenko/supplement}. We use SageMath, TOPCOM and Polymake for computations with the Newton polytope. We also use Fanosearch code for Magma (see~\cite{fanosearch/magma}) for the computation of Picard--Fuchs equations.

For Landau--Ginzburg models from Theorem~\ref{theorem:apery} (except the case \((D)\)) the associated Picard--Fuchs equation has degree \(2\) in \(s = t^{i_V}\). In fact, we use the following trick.

\begin{remark}
  Let \(V\) be a Fano complete intersection of index one of multidegree \((1, 1, \ldots, 1, j)\) in \(\Gr(k, N)\). Let \(W\) be a Fano complete intersection \((1, 1, \ldots, 1, 1)\) in \(\Gr(k, N)\) of the same codimension and of Fano index \(j\). Their regularised quantum periods are the same after the change of variables \(t \mapsto t^j\), and Picard--Fuchs equations can be obtained from one another as the pullback via \(t \mapsto t^j\) (see~\cite[Theorem~4.1]{przyjalkowski/grassmannian-planes}).
\end{remark}

\begin{proof}[Proof of Theorem~\ref{theorem:apery}]
  Computations from Appendix~\ref{appendix:laurent-polynomials} show that Landau--Ginzburg models from Subsection~\ref{subsection:polynomial-list} are amenable. Theorem~\ref{theorem:amenable} implies that they admit a (weak) log Calabi--Yau compactification and are tempered.

  Recall that for Grassmannians of planes \(\Gr(2, N)\) the dimension of its cohomology is \(b_{k} := \dim(H^k(\Gr(2, N))) = \lfloor k / 4 \rfloor + 1\) for even \(k\), and is zero otherwise. Note that the primitive cohomology group \(P^m(\Gr(2, N))\) is isomorphic to \(\mathbb{Q}\) for \(m \in 4 \mathbb{Z}\), where \(m \leqslant 2 N - 4\), and is zero otherwise. Actually, for any \(l \leqslant 2 N - 4\) we have \(b_{2 l} = \sum_{m \leqslant l} \dim(P^{2 m}(\Gr(2, N)))\), so \(\dim(P^{2 l}(\Gr(2, N))) = b_{2l} - b_{2 (l - 1)}\). As a consequence, we have a unique primitive cohomology class \(p_k\) in each even codimension. According to~\cite[\nopp 9.7]{galkin/gamma-mirror}, the Grassmannians \(\Gr(2, N)\) for \(N = 5,6,7\) have a principal Ap\'{e}ry constant \(\zeta(2)\) coming from the primitive class \(p_2\), while for \(N = 7\) we also have a principal Ap\'{e}ry constant \(\zeta(4)\) coming from the primitive class \(p_4\) (cf.~\cite{galkin/apery}). Moreover, note that \(c_1(E)^{\dim(\Gr) - 3} \cdot p_2 = c_1^{\dim(\Gr) - 7} (E) \cdot p_4 = 0\), so the class \(p_2\) restricts non-trivially for \((A)\)--\((D)\), while the class \(p_4\) vanishes on \((D)\). Here we denote by \(E\) the tautological bundle on \(\Gr(2, N)\), so \(c_1(E)\) generates the Picard group. It is also possible to compute the anticanonical quantum differential operator and its solutions via quantum Chevalley formula (see~\cite[Theorem~10.1]{fulton/quantum}) but we do not need to.

  Quantum Lefschetz principle (see~\cite[Section~8]{galkin/gamma-mirror}) then implies that there exists a principal Ap\'{e}ry constant \(\zeta(2)\) for \((A)\)--\((D)\), and a non-principal Ap\'{e}ry constant~\(\zeta(4)\) for the case \((D)\). Namely, it could be rephrased in terms of the Laplace transformation of Givental's J-functions (see~\cite[Lemma~8.1]{galkin/gamma-mirror}). Then anticanonical quantum \(\mathcal{D}\)-modules are also related by the Laplace transformation. Recall that the anticanonical quantum differential operator of a Fano variety coincides with the operator of minimal order and degree annihilating its quantum period (see Subsection~\ref{subsection:MS}). Then we only have to relate the corresponding quantum periods via~\cite[Theorem~F.1]{coates/periods}. Note that essentially the same argument was used in~\cite[Corollary~4.2]{golyshev/tate}.

  Computations from Appendix~\ref{appendix:local-systems} show that degree in \(t\) of the Picard--Fuchs equation of a Landau--Ginzburg model for \((A)\)--\((C)\) is \(d = 2 i_V\). Then there are no non-principal Ap\'{e}ry constants for \((A)\)--\((C)\) just for dimension reasons (see Remark~\ref{remark:recurrence}). Calculations of Appendix~\ref{appendix:local-systems} also show that the ramification defect of the local system~\(\mathcal{H}^3_v\) is 1, i.e., \(\rk(\IH^1(\mathbb{P}^1, \mathcal{H}_v^3)) = 1\) (see Subsection~\ref{subsection:MS}). Moreover, \(\IH^1(\mathbb{P}^1, \mathcal{H}_v^3)\) is pure of weight \(n = 4\), so \(\IH^1(\mathbb{P}^1, \mathcal{H}_v^3) \simeq \mathbb{Q}(-2)\), and there exists an exact sequence of mixed Hodge structures (see~\cite[Section~10]{kerr/unipotent} or~\cite[Section~6]{golyshev/apery})
  \[
    0 \rightarrow \IH^1(\mathbb{P}^1, \mathcal{H}^{n - 1}_v) \rightarrow
    \IH^1(\mathbb{P}^1 \setminus \{\infty\}, \mathcal{H}_v^{n - 1}) \rightarrow
    (\psi_{\infty} \mathcal{H}^{n - 1}_v)_{T_{\infty}}(-1) \rightarrow 0,
  \]
  where \(\psi_{\infty}(\cdot)\) is the limiting mixed HS functor at \(t = \infty\), and we put \((\cdot)_{T_{\infty}} = \coker(T_{\infty} - I)\).

  Calculations of Appendix~\ref{appendix:local-systems} and Euler--Poincar\'{e} theorem also imply that
  \begin{gather*}
    \rk(\IH^1(\mathbb{P}^1 \setminus \{\infty\}, \mathcal{H}_v^3)) =
    \sum_{\sigma \neq \infty} \rk(T_{\sigma} - I) - \rk(\mathcal{H}^3_v) \chi(\mathbb{A}^1) = \\
    = (\rf(\mathcal{H}^3_v) - 2 \rk(\mathcal{H}^3_v)) +
    (4 - \rk(T_{\infty} - I)) = 1 + \dim(\coker(T_{\infty} - I)).
  \end{gather*}
  For the cases \((A.1)\)--\((C.1)\) we have \(\coker(T_{\infty} - I) = 0\), so \(\IH^1(\mathbb{P}^1 \setminus \{\infty\}, \mathcal{H}_v^3) \simeq \mathbb{Q}(-2)\). For the case \((D)\) the monodromy at \(T_{\infty}\) is maximal unipotent, so \(\dim(\coker(T_{\infty} - I)) = 1\). The rank of \(\mathcal{H}^3_v\) is \(4\), so \(\coker(T_{\infty} - I) \simeq \mathbb{Q}(-3)\), and \(\rk(\IH^1(\mathbb{P}^1 \setminus \{\infty\}, \mathcal{H}_v^3)) = 2\).
\end{proof}

\begin{remark}\label{remark:calculations}
  Now let \(V\) be a smooth Fano variety, where \(n = \dim(V)\), and let us put \(R = d - \sum_{\sigma \neq 0,\infty} \rk(T_{\sigma} - I)\) (we continue to follow the notation introduced above). Assume also that \(\rk(\mathcal{H}^{n - 1}_v) = n\). Note that the assumption \(R = 0\) is a weak version of the ``normal conifold type'' assumption from~\cite[Remark~5.2]{golyshev/apery}. From the exact sequence of mixed Hodge structures mentioned above we obtain the identity
  \[
    d - 1 - R = \dim(\IH^1(\mathbb{P}^1, \mathcal{H}^{n - 1}_v)) + \dim(\coker(T_{\infty} - I)).
  \]
  If \(R = 0\), then we have \(\dim(\coker(T_{\infty} - I)) = d - 1 - \dim(\IH^1(\mathbb{P}^1, \mathcal{H}^{n - 1}_v))\): this happens for Fano threefolds in~\cite{golyshev/apery} and for \((A)\)--\((C)\) in Theorem~\ref{theorem:apery}. Finally, for \((D)\) we have \(d = 5\), \(\dim(\IH^1(\mathbb{P}^1, \mathcal{H}^{n - 1}_v)) = \dim(\coker(T_{\infty} - I)) = 1\), and \(R = 2\).
\end{remark}

\begin{proof}[Proof of Corollary~\ref{corollary:apery}]
  Theorem~\ref{theorem:apery} states that for the cases \((A.1)\)--\((C.1)\) the degree of Picard--Fuchs operator in \(t\) is \(2\), and the family \(\mathsf{f} \colon \mathcal{W} \rightarrow \mathbb{P}^1\) is tempered. Then we could apply~\cite[Theorem~10.10]{kerr/unipotent} to produce an admissible classical normal function on \(U = \mathbb{P}^1 \setminus \Sigma\) non-singular away from \(t = \infty\), where \(\Sigma\) is the singular locus of \(\mathcal{H}^3_v\). Moreover, by~\cite[Remark~5.2]{golyshev/apery} we could maximise the convergence radius by choosing among all truncated normal functions the one with the convergence radius \(\vert t_{\max} \vert\), where \(t_{\max}\) is a unique singular point of the local system with maximal absolute value (see Appendix~\ref{appendix:local-systems}). In other words, we choose a unique holomorphic solution to the inhomogeneous Picard--Fuchs equation analytic at \(0\) with no monodromy around the other singular point \(t_{\min}\). Consequently, we obtain a holomorphic function of the form \(f = \sum_{n = 0}^{\infty} (A_n f(0) - B_n) t^n\) (here we repeat the argument of the proof of~\cite[Theorem~10.1]{kerr/unipotent}).

  Our calculations show that \(\IH^1(\mathbb{P}^1 \setminus \{\infty\}, \mathcal{H}^3_v) \simeq \IH^1(\mathbb{P}^1, \mathcal{H}^3_v)\). In particular, it is a pure Hodge structure. Recall that the variety \(\mathcal{W}\) was constructed as an iterated blow-up of a toric variety in boundary strata (see Subsection~\ref{subsection:toric}), hence the Hodge conjecture for \(\mathcal{W}\) trivially holds. By Decomposition Theorem (see~\cite[\nopp 3.4,3.5]{kerr/exponential}) we could realise \(\IH^1(\mathbb{P}^1, \mathcal{H}^3_v)(2)\) as a \(\mathbb{Q}\)-Hodge substructure of \(H^4(\mathcal{W}, \mathbb{Q}(2))\). In particular, we realise the Hodge group \(\Hg(\IH^1(\mathbb{P}^1, \mathcal{H}^3_v)(2))_{\mathbb{Q}} := \Hom_{\mathrm{HS}}(\mathbb{Q}(0), \IH^1(\mathbb{P}^1, \mathcal{H}^3_v)(2))\) as a subgroup of the Hodge group of \(H^4(\mathcal{W}, \mathbb{Q}(2))\), and for the latter we do know the Hodge conjecture. Since \(\mathcal{H}_v^3\) has no global sections,~\cite[Corollary 2.21 and Proposition 3.3]{brosnan/singularities} identify admissible \(\mathbb{Q}\)-normal functions with values in \(\mathcal{H}_v^3(2)\) with the Hodge classes in its parabolic cohomology (through comparison of both these groups with the absolute Hodge cohomology group \(\IH^1_{\mathrm{AH}} (\mathbb{P}^1, \mathcal{H}^3_v(2))\), see also the exposition in~\cite[Lemmas~51--52]{kerr/exponential} for details):
\[
  \mathrm{NF}^{\mathrm{ad}}(U, \mathcal{H}_v^3(2))_{\mathbb{Q}} \simeq
  \Hom_{\mathrm{HS}} ( \mathbb{Q}(0), \IH^1(\mathbb {P}^1, \mathcal{H}_v^3(2))).
\]
We obtain that our admissible normal function comes from an algebraic 2-cycle on \(\mathcal{W}\).
\end{proof}

\clearpage
\appendix
\section{Landau--Ginzburg models: local systems}\label{appendix:local-systems}

In this appendix the notation \(\bullet_{\mu}^{(k, N)}\) means that \(\bullet\) corresponds to a complete intersection of multidegree \(\mu\) in \(\Gr(k, N)\).

\subsection{Regularised quantum periods}

\begin{align*}
  QP_{(1, 3)}^{(2, 5)} = 1 + 18 t + 1710 t^2 + 246960 t^3 + 43347150 t^4 + 8515775268 t^5 + \ldots \\
  QP_{(1, 1)}^{(2, 5)} = 1 + 18 t^3 + 1710 t^6 + 246960 t^9 + 43347150 t^{12} + 8515775268 t^{15} + \ldots \\
  QP_{(2, 2)}^{(2, 5)} = 1 + 12 t + 684 t^2 + 58800 t^3 + 6129900 t^4 + 714610512 t^5 + 89611475184 t^6 + \ldots \\
  QP_{(1, 2)}^{(2, 5)} = 1 + 12 t^2 + 684 t^4 + 58800 t^6 + 6129900 t^8 + 714610512 t^{10} + 89611475184 t^{12} + \ldots \\
  QP_{(1, 1, 1, 2)}^{(2, 6)} = 1 + 8 t + 288 t^2 + 15200 t^3 + 968800 t^4 + 68923008 t^5 + 5269660704 t^6 + \ldots \\
  QP_{(1, 1, 1, 1)}^{(2, 6)} = 1 + 8 t^2 + 288 t^4 + 15200 t^6 + 968800 t^8 + 68923008 t^{10} + 5269660704 t^{12} + \ldots \\
  QP_{(1, 1, 1, 1, 1, 1)}^{(2, 7)} = 1 + 5 t + 109 t^2 + 3317 t^3 + 121501 t^4 + 4954505 t^5 + 216867925 t^6 + \ldots
\end{align*}

\subsection{Picard--Fuchs operators}

Throughout this subsection \(D\) is the operator \(t \frac{\partial}{\partial t}\).
\begin{align*}
  PF_{(1, 3)}^{(2, 5)} = (27 t (27 t + 11) - 1) D^4 +
  54 t (54 t + 11) D^3 + \\ +
  3 t (1323 t + 148) D^2 +
  3 t (702 t + 49) D +
  18 t (20 t + 1); \\
  PF_{(1, 1)}^{(2, 5)} = (27 t^3 (27 t^3 + 11) - 1) D^4 +
  3 \cdot 54 t^3 (54 t^3 + 11) D^3 + \\ +
  3^2 \cdot 3 t^3 (1323 t^3 + 148) D^2 +
  3^3 \cdot 3 t^3 (702 t^3 + 49) D +
  3^4 \cdot 18 t^3 (20 t^3 + 1); \\
  PF_{(2, 2)}^{(2, 5)} = ((16 t (16 t + 11) - 1) D^4 +
  32 t (32 t + 11) D^3 + \\ +
  4 t (352 t + 67) D^2 +
  4 t (192 t + 23) D +
  12 t (12 t + 1)); \\
  PF_{(1, 2)}^{(2, 5)} = ((16 t^2 (16 t^2 + 11) - 1) D^4 +
  2 \cdot 32 t^2 (32 t^2 + 11) D^3 + \\ +
  2^2 \cdot 4 t^2 (352 t^2 + 67) D^2 +
  2^3 \cdot 4 t^2 (192 t^2 + 23) D +
  2^4 \cdot 12 t^2 (12 t^2 + 1)); \\
  PF_{(1, 1, 1, 2)}^{(2, 6)} = (8 t (54 t + 13) - 1) D^4 +
  16 t (108 t + 13) D^3 + \\ +
  6 t (406 t + 27) D^2 +
  2 t (708 t + 29) D +
  8 t (36 t + 1); \\
  PF_{(1, 1, 1, 1)}^{(2, 6)} = (8 t^2 (54 t^2 + 13) - 1) D^4 +
  2 \cdot 16 t^2 (108 t^2 + 13) D^3 + \\ +
  2^2 \cdot 6 t^2 (406 t^2 + 27) D^2 +
  2^3 \cdot 2 t^2 (708 t^2 + 29) D +
  2^4 \cdot 8 t^2 (36 t^2 + 1); \\
  PF_{(1, 1, 1, 1, 1, 1)}^{(2, 7)} =
  (t - 3)^2 \cdot (t^3 - 289 t^2 - 57 t + 1) D^4 + \\ +
  4 t \cdot (t - 3) \cdot (t^3 - 149 t^2 + 867 t + 85) D^3 + \\ +
  2 t \cdot (3 t^4 - 239 t^3 + 2353 t^2 - 7597 t - 408) D^2 + \\ +
  2 t \cdot (2 t^4 - 87 t^3 + 675 t^2 - 4773 t - 153) D + \\ +
  t (t^4 - 26 t^3 + 12 t^2 - 2166 t - 45).
\end{align*}
Note that the operators \(PF_{(1, 2)}^{(2, 5)}\), \(PF_{(1, 1)}^{(2, 5)}\), \(PF_{(1, 1, 1, 1)}^{(2, 6)}\) are pullbacks of \(PF_{(2, 2)}^{(2, 5)}\), \(PF_{(1, 3)}^{(2, 5)}\), \(PF_{(1, 1, 1, 2)}^{(2, 6)}\) under \(t \mapsto t^i\), correspondingly. Moreover, according to the database~\cite{almkvist/CY-tables} of fourth order differential operators of Calabi--Yau type, the operators \(PF_{(2, 2)}^{(2, 5)}\), \(PF_{(1, 3)}^{(2, 5)}\), \(PF_{(1, 1, 1, 2)}^{(2, 6)}\), \(PF_{(1, 1, 1, 1, 1, 1)}^{(2, 7)}\) have AESZ IDs 24,25,26,27, correspondingly.

\subsection{Ramification data}

None of the local systems is extremal: each has ramification defect equal to \(1\). In all cases there exists a unique singular point \(t_{\max}\) with maximal absolute value \(\vert t_{\max} \vert\) (unique up to multiplication of \(t_{\max}\) by roots of unity if the Fano index is larger than one). Below we provide Jordan normal forms of log-monodromy operators.

\subsubsection{CI of multidegree \((1, 3)\) in \(\Gr(2, 5)\)}
\(t_{\max} = -\frac{11}{54} - \frac{5 \sqrt{5}}{54}\),
\[
  \begin{pmatrix}
    0 & 1 & 0 & 0 \\
    0 & 0 & 1 & 0 \\
    0 & 0 & 0 & 1 \\
    0 & 0 & 0 & 0
  \end{pmatrix}_{t = 0}, \quad
  \begin{pmatrix}
    0 & 0 & 0 & 0 \\
    0 & 0 & 0 & 0 \\
    0 & 0 & 0 & 1 \\
    0 & 0 & 0 & 0
  \end{pmatrix}_{t^2 + \frac{11}{27} t - \frac{1}{729} = 0}, \quad
  \begin{pmatrix}
    2/3 & 0 & 0 & 0 \\
    0 & 2/3 & 0 & 0 \\
    0 & 0 & 1/3 & 0 \\
    0 & 0 & 0 & 1/3
  \end{pmatrix}_{t = \infty}.
\]

\subsubsection{CI of multidegree \((1, 1)\) in \(\Gr(2, 5)\)}
\(t_{\max}^3 = -\frac{11}{54} - \frac{5 \sqrt{5}}{54}\),
\[
  \begin{pmatrix}
    0 & 1 & 0 & 0 \\
    0 & 0 & 1 & 0 \\
    0 & 0 & 0 & 1 \\
    0 & 0 & 0 & 0
  \end{pmatrix}_{t = 0}, \quad
  \begin{pmatrix}
    0 & 0 & 0 & 0 \\
    0 & 0 & 0 & 0 \\
    0 & 0 & 0 & 1 \\
    0 & 0 & 0 & 0
  \end{pmatrix}_{t^6 + \frac{11}{27} t^3 - \frac{1}{729} = 0}, \quad
  \begin{pmatrix}
    0 & 0 & 0 & 0 \\
    0 & 0 & 0 & 0 \\
    0 & 0 & 0 & 0 \\
    0 & 0 & 0 & 0
  \end{pmatrix}_{t = \infty}.
\]

\subsubsection{CI of multidegree \((2, 2)\) in \(\Gr(2, 5)\)}
\(t_{\max} = -\frac{11}{32} - \frac{5 \sqrt{5}}{32}\),
\[
  \begin{pmatrix}
    0 & 1 & 0 & 0 \\
    0 & 0 & 1 & 0 \\
    0 & 0 & 0 & 1 \\
    0 & 0 & 0 & 0
  \end{pmatrix}_{t = 0}, \;
  \begin{pmatrix}
    0 & 0 & 0 & 0 \\
    0 & 0 & 0 & 0 \\
    0 & 0 & 0 & 1 \\
    0 & 0 & 0 & 0
  \end{pmatrix}_{t^2 + \frac{11}{16} t - \frac{1}{256} = 0}, \quad
  \begin{pmatrix}
    1/2 & 1 & 0 & 0 \\
    0 & 1/2 & 0 & 0 \\
    0 & 0 & 1/2 & 1 \\
    0 & 0 & 0 & 1/2
  \end{pmatrix}_{t = \infty}.
\]

\subsubsection{CI of multidegree \((1, 2)\) in \(\Gr(2, 5)\)}
\(t_{\max}^2 = -\frac{11}{32} - \frac{5 \sqrt{5}}{32}\),
\[
  \begin{pmatrix}
    0 & 1 & 0 & 0 \\
    0 & 0 & 1 & 0 \\
    0 & 0 & 0 & 1 \\
    0 & 0 & 0 & 0
  \end{pmatrix}_{t = 0}, \;
  \begin{pmatrix}
    0 & 0 & 0 & 0 \\
    0 & 0 & 0 & 0 \\
    0 & 0 & 0 & 1 \\
    0 & 0 & 0 & 0
  \end{pmatrix}_{t^4 + \frac{11}{16} t^2 - \frac{1}{256} = 0}, \quad
  \begin{pmatrix}
    0 & 1 & 0 & 0 \\
    0 & 0 & 0 & 0 \\
    0 & 0 & 0 & 1 \\
    0 & 0 & 0 & 0
  \end{pmatrix}_{t = \infty}.
\]

\subsubsection{CI of multidegree \((1^{(3)}, 2)\) in \(\Gr(2, 6)\)}
\(t_{\max} = -\frac{1}{4}\),
\begin{gather*}
  \begin{pmatrix}
    0 & 1 & 0 & 0 \\
    0 & 0 & 1 & 0 \\
    0 & 0 & 0 & 1 \\
    0 & 0 & 0 & 0
  \end{pmatrix}_{t = 0}, \;
  \begin{pmatrix}
    0 & 0 & 0 & 0 \\
    0 & 0 & 0 & 0 \\
    0 & 0 & 0 & 1 \\
    0 & 0 & 0 & 0
  \end{pmatrix}_{t = \frac{1}{108}}, \;
  \begin{pmatrix}
    0 & 0 & 0 & 0 \\
    0 & 0 & 0 & 0 \\
    0 & 0 & 0 & 1 \\
    0 & 0 & 0 & 0
  \end{pmatrix}_{t = -\frac{1}{4}}, \;
  \begin{pmatrix}
    2/3 & 0 & 0 & 0 \\
    0 & 1/2 & 0 & 0 \\
    0 & 0 & 1/2 & 0 \\
    0 & 0 & 0 & 1/3
  \end{pmatrix}_{t = \infty}.
\end{gather*}

\subsubsection{CI of multidegree \((1^{(4)})\) in \(\Gr(2, 6)\)}
\(t_{\max}^2 = -\frac{1}{4}\),
\[
  \begin{pmatrix}
    0 & 1 & 0 & 0 \\
    0 & 0 & 1 & 0 \\
    0 & 0 & 0 & 1 \\
    0 & 0 & 0 & 0
  \end{pmatrix}_{t = 0}, \;
  \begin{pmatrix}
    0 & 0 & 0 & 0 \\
    0 & 0 & 0 & 0 \\
    0 & 0 & 0 & 1 \\
    0 & 0 & 0 & 0
  \end{pmatrix}_{t^2 = \frac{1}{108}}, \;
  \begin{pmatrix}
    0 & 0 & 0 & 0 \\
    0 & 0 & 0 & 0 \\
    0 & 0 & 0 & 1 \\
    0 & 0 & 0 & 0
  \end{pmatrix}_{t^2 = -\frac{1}{4}}, \;
  \begin{pmatrix}
    0 & 0 & 0 & 0 \\
    0 & 0 & 0 & 0 \\
    0 & 0 & 2/3 & 0 \\
    0 & 0 & 0 & 1/3
  \end{pmatrix}_{t = \infty}.
\]

\subsubsection{CI of multidegree \((1^{(6)})\) in \(\Gr(2, 7)\)}
\(t_{\max} \approx 289.20\ldots\)
\[
  \begin{pmatrix}
    0 & 1 & 0 & 0 \\
    0 & 0 & 1 & 0 \\
    0 & 0 & 0 & 1 \\
    0 & 0 & 0 & 0
  \end{pmatrix}_{t = 0}, \quad
  \begin{pmatrix}
    0 & 0 & 0 & 0 \\
    0 & 0 & 0 & 0 \\
    0 & 0 & 0 & 1 \\
    0 & 0 & 0 & 0
  \end{pmatrix}_{t^3 - 289 t^2 - 57 t + 1 = 0}, \quad
  \begin{pmatrix}
    0 & 1 & 0 & 0 \\
    0 & 0 & 1 & 0 \\
    0 & 0 & 0 & 1 \\
    0 & 0 & 0 & 0
  \end{pmatrix}_{t = \infty}.
\]

\section{Landau--Ginzburg models: Laurent polynomials}\label{appendix:laurent-polynomials}

\subsection{Laurent polynomials defining LG models}\label{subsection:polynomial-list}

The following Laurent polynomials are either taken from~\cite{przyjalkowski/grassmannians,przyjalkowski/grassmannian-planes}, or are obtained by \emph{Przyjalkowski method} (for example, see~\cite{coates/toric}), and also composed with some mutation (see the computations in~\cite{ovcharenko/supplement}).
\begin{align*}
  LG_{(1, 3)}^{(2, 5)} = (b^{-1} c^{-1} d) \cdot (c + 1) \cdot (c + 1 + b a^{-1}) \cdot (a + b + 1 + d^{-1})^3; \\
  LG_{(1, 1)}^{(2, 5)} = a_{22} a_{11}^{-1} + a_{22} a_{21}^{-1} + a_{31}^{-1} + a_{31} a_{21}^{-1} + a_{22}^{-1} + a_{21} + a_{11}; \\
  LG_{(2, 2)}^{(2, 5)} = (a_{11}^{-4} a_{22}^{-2} a_{31}^{-2} a_{12}^{-1}) \cdot (a_{11} a_{22} + a_{11} a_{31} + a_{31}) \cdot \\ \cdot
  (a_{11} a_{22} + a_{11} a_{31} + a_{11} + 1)^2 \cdot (a_{11}^2 a_{22} a_{31} a_{12} + 1)^2; \\
  LG_{(1, 2)}^{(2, 5)} = (a_{11} a_{31} (a_{12} + 1) (a_{12} + a_{21}) + a_{12} + a_{21} + a_{21} a_{12}^{-1})^2 a_{11}^{-1} + (a_{12} a_{21} a_{31})^{-1}; \\
  LG_{(1, 1, 1, 2)}^{(2, 6)} = (x (w y z + z + 1) \cdot (x (z + 1) + y z + 1) + y z (w x z + 1))^2 \cdot (w x^2 y^2 z^3)^{-1}; \\
  LG_{(1, 1, 1, 1)}^{(2, 6)} = (a_{41} + a_{31} + a_{21} + a_{11}) \cdot (1 + a_{31}^{-1} a_{21}^{-1} a_{11}^{-1} \cdot (a_{41} + a_{31})) + a_{31}^{-1} + a_{41}^{-1}; \\
  LG_{(1^{(6)})}^{(2, 7)} = (a_{51} (a_{21}^2 a_{31} (a_{41} + 1) + a_{21} + 1) + a_{21}) \cdot (a_{21}^4 a_{31} a_{41} a_{51})^{-1} \cdot \\ \cdot
  ((a_{21}^2 a_{31} (a_{41} + 1) + 1) (a_{51} (a_{21}^2 a_{31} (a_{41} + 1) + a_{21} + 1) + a_{21}^2 a_{41} + a_{21}) + a_{21}^2) .
\end{align*}

\subsection{Fine regular star triangulations of the dual polytope giving smooth toric crepant resolutions}

In this subsection columns of the matrix \(v\) represent the ordered vertices of the dual polytope \(P^{\vee}\) except for the first column \(v_0\), which is always chosen to be the unique integral interior point. Similarly, columns of the matrix \(T\) describe the corresponding simplex in the chosen fine regular star triangulation, where entries correspond to columns of the matrix \(v\). Validation of smoothness of the corresponding toric variety can be done by explicit computations (see~\cite{ovcharenko/supplement}). In all of our triangulations the star origin is chosen to be the first column \(v_0\), i.e., the unique interior lattice point.
\begin{gather*}
  v_{(1, 3)}^{(2, 5)} =
  \begin{pmatrix}
    0 & -1 & 0 & 0 & 0 & 0 & 1 & 1 & 1 \\
    0 & -1 & 0 & 0 & 0 & 1 & 0 & 0 & 1 \\
    0 & 0 & -1 & 0 & 1 & 0 & -1 & 0 & 0 \\
    0 & 1 & 0 & -1 & 0 & 0 & 0 & 0 & 0
  \end{pmatrix}; \\
  v_{(2, 2)}^{(2, 5)} =
  \begin{pmatrix}
    0 & -1 & 0 & 0 & 0 & 0 & 0 & 0 & 1 & 1 \\
    0 & 0 & -1 & 0 & 0 & 0 & 1 & 1 & -1 & -1 \\
    0 & 0 & -1 & 0 & 0 & 1 & 0 & 1 & -1 & 0 \\
    0 & 2 & 2 & -1 & 1 & -1 & -1 & -2 & 0 & -1
  \end{pmatrix}; \\
  v_{(1, 1, 1, 2)}^{(2, 6)} =
  \begin{pmatrix}
    0 & -1 & 0 & 0 & 0 & 0 & 0 & 0 & 1 & 1 & 1 & 1 \\
    0 & -1 & -1 & -1 & 0 & 0 & 0 & 1 & 0 & 0 & 1 & 1 \\
    0 & 0 & 0 & 1 & -1 & 0 & 0 & 0 & -1 & 0 & -1 & 0 \\
    0 & 1 & 1 & 0 & 0 & -1 & 1 & -1 & 0 & -1 & 0 & -1
  \end{pmatrix}; \\
  v_{(1, 1, 1, 1, 1, 1)}^{(2, 7)} =
  \begin{pmatrix}
    0 & -1 & -1 & -1 & -1 & 0 & 0 & 0 & 0 & 0 & 0 & 1 & 1 & 1 & 1 \\
    0 & -1 & 0 & 0 & 0 & -1 & 0 & 0 & 0 & 1 & 1 & -1 & 0 & 0 & 0 \\
    0 & 1 & 0 & 1 & 2 & 1 & 1 & 1 & 2 & 0 & 1 & -1 & -2 & -1 & 0 \\
    0 & 0 & 0 & -1 & -1 & 0 & -1 & 0 & -1 & 0 & -1 & 1 & 1 & 0 & 0
  \end{pmatrix}.
\end{gather*}

\[
  T_{(1, 3)}^{(2, 5)} =
  \begin{pmatrix}
    0 & 0 & 0 & 0 & 0 & 0 & 0 & 0 & 0 & 0 & 0 & 0 & 0 & 0 & 0 \\
    1 & 1 & 1 & 1 & 1 & 1 & 1 & 1 & 1 & 1 & 2 & 3 & 3 & 3 & 3 \\
    2 & 2 & 2 & 3 & 3 & 3 & 4 & 4 & 5 & 6 & 3 & 4 & 4 & 5 & 6 \\
    3 & 3 & 5 & 4 & 4 & 6 & 5 & 7 & 6 & 7 & 5 & 5 & 7 & 6 & 7 \\
    5 & 6 & 6 & 5 & 7 & 7 & 8 & 8 & 8 & 8 & 6 & 8 & 8 & 8 & 8
  \end{pmatrix};
\]
\[  
  T_{(2, 2)}^{(2, 5)} =
  \begin{pmatrix}
    0 & 0 & 0 & 0 & 0 & 0 & 0 & 0 & 0 & 0 & 0 & 0 & 0 & 0 & 0 & 0 & 0 & 0 & 0 & 0 \\
    1 & 1 & 1 & 1 & 1 & 1 & 1 & 1 & 2 & 2 & 2 & 2 & 2 & 2 & 3 & 3 & 3 & 4 & 4 & 4 \\
    2 & 2 & 2 & 2 & 3 & 3 & 4 & 4 & 3 & 3 & 3 & 4 & 4 & 4 & 5 & 6 & 7 & 5 & 6 & 7 \\
    3 & 3 & 4 & 4 & 5 & 6 & 5 & 6 & 5 & 6 & 8 & 5 & 6 & 8 & 7 & 7 & 8 & 7 & 7 & 8 \\
    5 & 6 & 5 & 6 & 7 & 7 & 7 & 7 & 9 & 8 & 9 & 9 & 8 & 9 & 9 & 8 & 9 & 9 & 8 & 9
  \end{pmatrix};
\]
\begin{align*}
  T_{(1, 1, 1, 2)}^{(2, 6)} =
  \begin{pmatrix}
    0 & 0 & 0 & 0 & 0 & 0 & 0 & 0 & 0 & 0 & 0 & 0 & 0 & 0 & \cdots \\
    1 & 1 & 1 & 1 & 1 & 1 & 1 & 1 & 2 & 2 & 2 & 3 & 3 & 3 & \cdots \\
    2 & 2 & 2 & 3 & 3 & 3 & 4 & 4 & 3 & 3 & 4 & 4 & 5 & 5 & \cdots \\
    3 & 3 & 4 & 4 & 5 & 6 & 5 & 6 & 4 & 6 & 6 & 5 & 7 & 8 & \cdots \\
    4 & 6 & 6 & 5 & 7 & 7 & 7 & 7 & 8 & 8 & 8 & 8 & 9 & 9 & \cdots
  \end{pmatrix}
  \\
  \begin{pmatrix}
    \cdots & 0 & 0 & 0 & 0 & 0 & 0 & 0 & 0 & 0 & 0 & 0 & 0 & 0 & 0 \\
    \cdots & 3 & 3 & 3 & 3 & 4 & 4 & 4 & 4 & 5 & 6 & 6 & 6 & 7 & 7 \\
    \cdots & 6 & 6 & 6 & 7 & 5 & 6 & 6 & 7 & 7 & 7 & 8 & 9 & 8 & 9 \\
    \cdots & 7 & 8 & 9 & 9 & 7 & 7 & 8 & 8 & 8 & 10 & 9 & 10 & 9 & 10 \\
    \cdots & 11 & 9 & 11 & 11 & 8 & 10 & 10 & 10 & 9 & 11 & 10 & 11 & 10 & 11
  \end{pmatrix};
\end{align*}
\begin{align*}
  T_{(1, 1, 1, 1, 1, 1)}^{(2, 7)} =
  \begin{pmatrix}
    0 & 0 & 0 & 0 & 0 & 0 & 0 & 0 & 0 & 0 & 0 & 0 & 0 & 0 & \cdots \\
    1 & 1 & 1 & 1 & 1 & 1 & 1 & 1 & 2 & 2 & 2 & 2 & 2 & 2 & \cdots \\
    2 & 2 & 2 & 2 & 3 & 3 & 4 & 5 & 3 & 3 & 3 & 4 & 7 & 9 & \cdots \\
    3 & 3 & 4 & 7 & 4 & 5 & 5 & 7 & 4 & 9 & 11 & 7 & 9 & 11 & \cdots \\
    4 & 11 & 7 & 11 & 5 & 11 & 7 & 11 & 9 & 12 & 12 & 9 & 11 & 12 & \cdots
  \end{pmatrix}
  \\
  \begin{pmatrix}
    \cdots & 0 & 0 & 0 & 0 & 0 & 0 & 0 & 0 & 0 & 0 & 0 & 0 & 0 & 0 & \cdots \\
    \cdots & 3 & 3 & 3 & 3 & 3 & 3 & 3 & 3 & 3 & 4 & 4 & 4 & 4 & 4 & \cdots \\
    \cdots & 4 & 4 & 4 & 5 & 6 & 6 & 9 & 10 & 11 & 5 & 5 & 6 & 7 & 8 & \cdots \\
    \cdots & 5 & 6 & 9 & 6 & 10 & 11 & 10 & 12 & 12 & 6 & 7 & 8 & 8 & 9 & \cdots \\
    \cdots & 6 & 10 & 10 & 11 & 13 & 13 & 12 & 13 & 13 & 8 & 8 & 10 & 9 & 10 & \cdots
  \end{pmatrix}
  \\
  \begin{pmatrix}
    \cdots & 0 & 0 & 0 & 0 & 0 & 0 & 0 & 0 & 0 & 0 & 0 & 0 & 0 & 0 \\
    \cdots & 5 & 5 & 6 & 6 & 7 & 7 & 7 & 8 & 8 & 8 & 9 & 9 & 10 & 11 \\
    \cdots & 6 & 7 & 8 & 8 & 8 & 8 & 9 & 9 & 10 & 11 & 10 & 11 & 12 & 12 \\
    \cdots & 8 & 8 & 10 & 11 & 9 & 11 & 11 & 10 & 13 & 13 & 12 & 12 & 13 & 13 \\
    \cdots & 11 & 11 & 13 & 13 & 14 & 14 & 14 & 14 & 14 & 14 & 14 & 14 & 14 & 14
  \end{pmatrix}.
\end{align*}

\subsection{Facet polynomials}

\subsubsection{CI of multidegree \((1, 3)\) in \(\Gr(2, 5)\)}

\begin{multicols}{2}
  \begin{enumerate}
  \item \((X_2 + 1) \cdot (X_1 X_2 + X_1 + 1) \cdot \\ \cdot (X_0 X_1 + X_0 + 1)^3\).
  \item \((X_2 + 1) \cdot (X_0 X_2 + X_0 + 1) \cdot \\ \cdot (X_0 X_1 + X_1 + 1)^3\).
  \item \((X_2 + 1)^2 \cdot (X_0 + X_1 + 1)^3\).
  \item \((X_0 + X_1 + X_2 + 1)^3\).
  \item \((X_2 + 1) \cdot (X_0 + X_1 + 1)^3\).
  \item \((X_2 + 1) \cdot (X_0 + X_1 + 1)^3\).
  \item \((X_0 + 1) \cdot (X_0 X_2 + X_1 + X_2 + 1)^3\).
  \item \((X_1 + 1) \cdot (X_0 + 1)^3 \cdot (X_1 X_2 + X_2 + 1)\).
  \end{enumerate}
\end{multicols}

\subsubsection{CI of multidegree \((1, 1)\) in \(\Gr(2, 5)\)}

\begin{multicols}{2}
  \begin{enumerate}
  \item \(X_0 X_1 + X_0 + X_1 + X_2 + 1\).
  \item \(X_0 + X_1 + X_2 + X_0^{-1} X_1 X_2 + 1\).
  \item \(X_0 + X_1 + X_2 + X_0^{-1} X_1 X_2 + 1\).
  \item \(X_1 X_2 + X_0 + X_1 + X_2 + 1\).
  \item \(X_0 + X_1 + X_2 + 1\).
  \item \(X_0 + X_1 + X_2 + 1\).
  \item \(X_0 + X_1 + X_2 + 1\).
  \item \(X_0 + X_1 + X_2 + 1\).
  \end{enumerate}
\end{multicols}

\clearpage
\subsubsection{CI of multidegree \((2, 2)\) in \(\Gr(2, 5)\)}

\begin{multicols}{2}
  \begin{enumerate}
  \item \((X_0 + X_2 + 1) \cdot \\ \cdot (X_0 + X_1 + X_0^{-1} X_1 X_2 + 1)^2\).
  \item \((X_1 + 1) \cdot (X_2 + 1)^2 \cdot (X_0 X_1 + X_0 + 1)^2\).
  \item \((X_2 + 1)^2 \cdot (X_0 + 1)^2 \cdot (X_0 X_1 + X_1 + 1)\).
  \item \((X_2 + 1)^2 \cdot (X_0 + X_1 + 1)^2\).
  \item \((X_0 + 1) \cdot (X_2 + 1)^2 \cdot (X_1 + 1)^2\).
  \item \((X_0 + X_1 + 1) \cdot (X_2 + 1)^2 \cdot (X_0 + 1)^2\).
  \item \((X_0 + 1) \cdot (X_2 + 1)^2 \cdot (X_0 + X_1 + 1)^2\).
  \item \((X_2 + 1)^2 \cdot (X_0 + X_1 + 1)^2\).
  \item \((X_0 + X_2 + 1) \cdot \\ \cdot (X_0 + X_1 + X_0^{-1} X_1 X_2 + 1)^2\).
  \end{enumerate}
\end{multicols}

\subsubsection{CI of multidegree \((1, 2)\) in \(\Gr(2, 5)\)}

\begin{multicols}{2}
  \begin{enumerate}
  \item \(((X_0 + X_1) (X_1 X_2 + X_1 + X_2) + X_1)^2\).
  \item \(((X_0 + X_2) (X_1 + X_2) X_2^{-1})^2 + X_2\).
  \item \(((X_0 + X_2) (X_1 + X_2) X_2^{-1})^2 + X_2\).
  \item \((X_0 + X_1 + X_2)^2 + X_0\).
  \item \(((X_0 + 1) (X_2 + 1))^2 + X_1\).
  \item \((X_0 + X_2 + 1)^2 + X_1\).
  \item \((X_1 + X_2 + 1)^2 + X_0\).
  \item \((X_1 + X_2 + 1)^2 + X_0\).
  \end{enumerate}
\end{multicols}

\subsubsection{CI of multidegree \((1^{(3)}, 2)\) in \(\Gr(2, 6)\)}

\begin{multicols}{2}
  \begin{enumerate}
  \item \((X_0 + X_1 + 1)^2 \cdot (X_0 X_2 + X_2 + 1)^2\).
  \item \((X_0 + 1)^2 \cdot ((X_1 + 1) \cdot (X_2 + 1) + X_0 X_2)^2\).
  \item \((X_0 X_2 + X_0 + X_1 + X_2 + 1)^2\).
  \item \((X_1 X_2 + X_0 + X_1 + X_2 + 1)^2\).
  \item \((X_0 X_1 + X_0 + X_1 + X_2 + 1)^2\).
  \item \(((1 + X_2 (X_1 + 1)) \cdot (1 + X_0 (X_1 + 1)))^2\).
  \item \((X_0 + X_1 + X_2 + 1)^2\).
  \item \((1 + X_0^{-1} (X_0 + X_1) \cdot(X_0 + X_2))^2\).
  \item \((X_0 + X_1 + X_2 + 1)^2\).
  \item \((X_0 + X_1 + X_2 + 1)^2\).
  \item \(((X_1 + 1) \cdot (X_0 + X_2 + 1) + X_0^{-1} X_1 X_2)^2\).
  \end{enumerate}
\end{multicols}

\subsubsection{CI of multidegree \((1^{(4)})\) in \(\Gr(2, 6)\)}

\begin{multicols}{2}
  \begin{enumerate}
  \item \((X_1 + X_2 + 1) \cdot (X_0 X_1 + X_0 + 1)\).
  \item \((X_1 + X_2 + 1) \cdot (X_0 X_1 + X_0 + 1)\).
  \item \(X_0 + X_2 + 1 + \\ + X_0^{-1} X_1 (X_0 + 1)(X_0 + X_2)\).
  \item \((X_0 X_1 + X_0 X_1 X_2^{-1} + X_0 + X_1 + 1) \cdot \\ \cdot (X_2 + 1)\).
  \item \(X_0 + X_1 + X_2 + 1\).
  \item \(X_0 + X_1 + X_2 + X_0^{-1} X_1 X_2 + 1\).
  \item \(X_0 X_1 + X_0 + X_1 + X_2 + 1\).
  \item \(X_0 X_1 + X_0 + X_1 + X_2 + 1\).
  \item \(X_0 + X_1 + X_2 + X_0^{-1} X_1 X_2 + 1\).
  \item \(X_0 + X_1 + X_2 + 1\).
  \item \(X_0 + X_1 + X_2 + 1\).
  \end{enumerate}
\end{multicols}

\subsubsection{CI of multidegree \((1^{(6)})\) in \(\Gr(2, 7)\)}

\begin{multicols}{2}
  \begin{enumerate}
  \item \((X_1 + X_2 + 1) \cdot (X_0 + X_1 + X_2 + 1)^2\).
  \item \((X_0 + 1) \cdot (X_1 (X_0 + 1) + 1) \cdot \\ \cdot (1 + X_2 (1 + X_1 (X_0 + 1)))\).
  \item \((X_0 + 1) \cdot (X_2 (X_0 + 1) + 1) \cdot \\ \cdot (X_2 (X_0 + 1) + X_1 + 1)\).
  \item \((X_0 + X_1 + 1) \cdot (1 + X_2 (X_0 + X_1 + 1))^2\).
  \item \((X_1 + 1) \cdot (X_0 + X_1 + 1) \cdot \\ \cdot (X_2 (X_0 + X_1 + 1) + 1)\).
  \item \((X_1 + 1) \cdot (X_2 (X_1 + 1) + 1) \cdot \\ \cdot (X_2 (X_1 + 1) + X_0 + 1)\).
  \item \((X_0 + 1) \cdot (X_2 (X_0 + 1) + X_1 + 1)\).
  \item \((X_0 + 1) \cdot (X_1 (X_0 + 1) + X_2 + 1)\).
  \item \((X_0 + X_1 + 1) \cdot (1 + X_0^{-1} X_1 X_2) + X_2\).
  \item \((X_2 + 1) \cdot (X_1 + 1) \cdot (X_0 + 1)\).
  \item \((X_0 + X_1 + 1) \cdot (X_2 (X_1 + 1) + 1) \cdot \\ \cdot (X_2 (X_0 + X_1 + 1) + 1)\).
  \item \((X_1 (X_0 + 1) + 1) \cdot \\ \cdot ((X_0 + 1) \cdot (X_1 + 1) + X_2)\).
  \item \((X_1 (X_0 + 1) + 1) \cdot \\ \cdot (X_0 X_1 (X_0 + 1) + X_0 + X_2 + 1)\).
  \item \((X_1 + 1) \cdot (X_0 + X_2 + 1) + X_0^{-1} X_1 X_2\).
  \end{enumerate}
\end{multicols}

\clearpage
\printbibliography

\end{document}